\documentclass[11pt]{amsart}
\usepackage[utf8]{inputenc}

\usepackage{amsmath,amssymb,amsmath}
\usepackage{mathabx}
\usepackage{graphicx}
\usepackage{subfigure}
\usepackage{cite}
\usepackage{hyperref}
\usepackage{url}
\usepackage{mathabx}
\usepackage{enumitem}
\usepackage{mathrsfs}  

\numberwithin{equation}{section}

\newcommand{\abs}[1]{\left\vert#1\right\vert}
\newcommand{\norm}[1]{\left\Vert#1\right\Vert}
\newcommand{\paren}[1]{\left(#1\right)}
\newcommand{\bracket}[1]{\left[#1\right]}
\newcommand{\set}[1]{\left\{#1\right\}}

\newcommand{\inner}[2]{\left\langle #1,#2\right\rangle}
\newcommand{\trinorm}[1]{\left\vvvert#1\right\vvvert}

\newcommand{\E}{\mathbb{E}}
\newcommand{\R}{\mathbb{R}}

 \newcommand{\N}{\mathbb{N}}
\newcommand{\Dt}{\Delta t}

\newcommand{\eps}{\epsilon}

\newcommand{\bP}{\mathbf{P}}
\newcommand{\bI}{\mathbf{I}}

\newcommand{\nd}{ [n]^d}

\newcommand{\1}{{\mathbf 1}}

\newcommand{\bi}{{\bar{i}}}
\newcommand{\bj}{{\bar{j}}}

\newcommand{\bQ}{\mathbf{Q}}
\newcommand{\bK}{\mathbf{K}}
\newcommand{\bN}{\mathbf{N}}
\newcommand{\bJ}{\mathbf{J}}

\newcommand{\cH}{\mathcal{H}}
\newcommand{\cF}{\mathcal{F}}

\newcommand{\HpT} {\mathscr{H}_T^p}
\newcommand{\scH} {\mathscr{H}}

\renewcommand{\P}{\mathbb{P}}

\DeclareMathOperator{\bigo}{O}
\DeclareMathOperator{\littleo}{o}

\DeclareMathOperator{\Tr}{Tr}

\DeclareMathOperator{\Lip}{Lip}

\DeclareMathOperator*{\esssup}{ess\,sup}

\newtheorem{theorem}{Theorem}[section]
\newtheorem{proposition}[theorem]{Proposition}
\newtheorem{lemma}[theorem]{Lemma}
\newtheorem{corollary}[theorem]{Corollary}
\newtheorem{remark}[theorem]{Remark}
\newtheorem{definition}[theorem]{Definition}
\newtheorem{example}[theorem]{Example}

\usepackage[dvipsnames]{xcolor}

\newcommand{\rev}[1]{#1}

\title[Nonlocal Diffusion with Noise]{A Numerical Method for a
  Nonlocal Diffusion Equation with Additive Noise}
\author{Georgi Medvedev}
\author{Gideon Simpson}
\date{\today}

\begin{document}

\maketitle
\begin{abstract}
We consider a nonlocal evolution equation representing the continuum
  limit of a large ensemble of interacting particles on graphs forced by noise. The two
  principle ingredients of the continuum model are a nonlocal term and Q-Wiener process 
  describing the interactions among the particles in the network and stochastic forcing
  respectively. The network connectivity is given by a square integrable function 
  called a graphon.

  We prove that the initial value problem for the continuum model
  is well-posed. Further, we construct a semidiscrete (discrete in space and continuous in time)
  and a fully discrete schemes for the nonlocal model. The former is obtained by a discontinuous
  Galerkin method and the latter is based on further discretizing time using the Euler-Maruyama
  method. We prove convergence and estimate the rate of convergence in each case.
  For the semidiscrete scheme, the rate of convergence estimate is expressed in terms of the regularity
  of the graphon, Q-Wiener process, and the initial data. We work in generalized Lipschitz spaces,
  which allows to treat models with data of lower regularity. This is important
  for applications as many interesting types of connectivity including small-world and power-law
  are expressed by graphons that are not smooth. The error analysis of the fully discrete scheme,
  on the other hand, reveals that for some models common in applied science, one has a higher
  speed of convergence than that predicted by the standard estimates for the Euler-Maruyama
  method. The rate of convergence analysis is supplemented with detailed numerical experiments,
  which are consistent with our analytical results.

  As a by-product, this work presents a rigorous justification for
  taking continuum limit for a large class of interacting dynamical systems on graphs subject to noise.
 \end{abstract}

  \section{Introduction}

  \subsection{The model}
In this work, we study an initial value problem (IVP) for the following
stochastically forced nonlocal evolution equation
\begin{subequations}
\label{e:kuramoto}
\begin{align}
\begin{split}
  du(t,x) &= \left\{ f(t,u) + \int K(x,y) S(u(t,x), u(t,y)) dy\right\}dt + dW(t,x),
\end{split}\\
  \label{e:kuramoto-ic}
  u(0,x)&=g(x),
\end{align}
\end{subequations}
where $u(t,x)$ is a real-valued  process defined on $[0,T]\times I^d$
with $T>0$ being an arbitrary but fixed time horizon and $I:=[0,1]$ throughout this paper.
The Gaussian process $W(t,x)$ will be defined below. The domain of integration
on the right--hand side of \eqref{e:kuramoto} is implicitly assumed to be 
$I^d$. The same convention will be used every time the spatial domain of integration  is not specified.

Equation \eqref{e:kuramoto} is a phenomenological
model of a continuous population of interacting particles subject to stochastic forcing. Function $f\left(t,u(t,x)\right)$
defines the intrinsic dynamics of a given particle at point $x\in I^d$  and time $t>0$, while the integral term on the right
hand side of \eqref{e:kuramoto} describes the interaction with other particles in the population. Here, the 
function $S(u(t,x), u(t,y))$ models pairwise interactions between particles located at $x\in I^d$ and $y\in I^d$ and
a measurable bounded $K(x,y)$ describes spatial connectivity between particles.

One way to arrive at a model of the form \eqref{e:kuramoto} is from the continuum limit of a dynamical system for a discrete population
of interacting particles \cite{Med14a, Med14b}. The continuous Kuramoto model of coupled phase oscillators
\cite{Med19, Laing17}
and neural fields \cite{CGP14} are two prominent examples of models of this type. Another class of models
leading to \eqref{e:kuramoto} are nonlocal diffusion equations \cite{AMRT10} including nonlinear and fractional diffusion
models \cite{Vazquez17, TEJ17, PabQuiRod16}. Other examples include models in population dynamics
\cite{BouCal10, CarFife05}, swarming \cite{MotTad14}, and peridynamics \cite{Du2018}, to name a few.

We complete the formulation of \eqref{e:kuramoto} by specifying 
assumptions on $f, K,$ and $S$. We assume that $f:[0,T]\times \R\to\R$   satisfies a linear growth bound and a Lipschitz condition:
\begin{subequations}
\label{e:fAssumptions}
\begin{align}
    \label{e:fBound}
    |f(t,u)| &\leq A_f + B_f |u|,\\
    \label{e:fLip}
    |f(t,u) - f(t',u')|&\leq L_f(|t-t'| + |u-u'|),
\end{align}
\end{subequations}
with positive constants $A_f$, $B_f$, and $L_f$.
$S:\R^2\to \R$ also  satisfies linear growth and Lipschitz conditions
\begin{subequations}
\label{e:SAssumptions}
\begin{align}
\label{e:SBound}
    |S(u,v)| & \leq  A_S + B_S(|u| + |v|),\\
\label{e:SLip}
    |S(u,v) - S(u',v')|&\leq L_S (|u-u'| + |v-v'|).
\end{align}
\end{subequations}
Again, $A_S$, $B_S$, and $L_S$ are positive constants.
For the interaction kernel, it will be necessary to assume both
\begin{subequations}
\label{e:Kbounds}
\begin{align}
   K_1 & \equiv \esssup_{x\in I^d}\rev{\int} |K(x,y)|^2 dy <\infty,\\
   K_2 & \equiv \esssup_{y\in I^d}\rev{\int} |K(x,y)|^2 dx <\infty.
\end{align}
\end{subequations}

Finally, we define $W(t,x)$. Let $\mathbf{Q}$ be a positive self-adjoint trace class operator
on \rev{$\cH:=L^2(I^d)$}. Let $\lambda_k, k\in\N,$ denote the eigenvalues of $\bQ$ arranged in the
decreasing order, counting multiplicity, and let $e_k\in \cH$ be the corresponding orthonormal
eigenfunctions. Then $W$,  a $Q$-Wiener Gaussian process is given via its Karhunen-Lo\`eve expansion as
\begin{equation}\label{QWien}
  W(t,x)=\sum_{k=1}^\infty \sqrt{\lambda_k} e_k(x) B_k(t),
\end{equation}
where the $B_k(t), k\in\N,$ are independent Brownian motions.

\subsection{The Galerkin approximation}

We next introduce a continuous in time Galerkin discretization of \eqref{e:kuramoto}.  First, the  domain $V=I^d$ is partitioned as
\begin{equation}
\label{e:Vcell}
\begin{split}
V^n_{\bi} &= (x_{i_1-1},x_{i_1}]\times (x_{i_2-1},x_{i_2}]\times\dots
\times (x_{i_d-1},x_{i_d}], \\ 
\bi&=(i_1,i_2, \dots, \rev{i_d})\in [n]^d,
\end{split}
\end{equation}
where
\begin{equation}
    \label{e:OneDMesh}
    x_i= ih, \; h=n^{-1},\; i\in\{0,1,\dots, n\}.
  \end{equation}
Next, the Galerkin basis is defined as 
\begin{equation}
\label{e:Hnsubspace}
    \mathcal{H}^n=\{\chi^n_{\bi}(x),\;  \bi\in
    [n]^d\},\quad \chi^n_{\bi}(x)  :=\1_{V^n_{\bi}}(x),
  \end{equation}
  where $\1_A$ is the indicator function of set $A$.
  Substituting
  \begin{equation}\label{step-u}
    u^n(t,x)=\sum_{\bi\in [n]^d} u^n_\bi (t)\chi_\bi^n(x),
  \end{equation}
  into \eqref{e:kuramoto}, and  projecting with respect to $L^2$  onto \rev{$\mathcal{H}^n$},
  we  arrive that the following semidiscrete IVP
\begin{subequations}
\label{e:finite_kuramoto}
\begin{align}
  du^n_{\bi} &=\Big\{f(t, u^n_{\bi}) + h^d\sum_{\bj\in [n]^d}  K^n_{\bi\bj} S(u^n_{\bi},u^n_{\bj}) \Big\} dt+ dW^n_{\bi}\\
  \label{e:finite_kuramoto-ic}
  u^n_{\bi}(0)& = g^n_{\bi}, \; \bi\in \rev{[n]^d},
\end{align}
\end{subequations}
where 
\begin{subequations}
\begin{align}
\label{e:Knij}
    K^n_{\bi\bj} & = h^{-2d} \iint K(x,y)\chi^n_{\bi}(x)\chi^n_{\bj}(y)dxdy,\\
    W^n_{\bi}(t) & = h^{-d} \inner{W(t,\cdot)}{\chi^n_\bi},\\
    g^n_{\bi} &= h^{-d} \inner{g}{\chi^n_\bi}.
\end{align}
\end{subequations}
Here, $\inner{\cdot}{\cdot}$ stands for the inner product of $\mathcal{H} = L^2(I^d)$.  \rev{The double integral in \eqref{e:Knij} is over $I^d\times I^d$; again, unless otherwise indicated, such double integrals will be over this set.}

\subsection{Organization and main results}

We begin our study by establishing well-posedness of the IVP \eqref{e:kuramoto}. This is the content of
Theorem~\ref{thm:kuramoto} and subsequent results in Section \ref{s:IVP}.  A fundamental challenge in
studying \eqref{e:kuramoto} is that the nonlocal term does not introduce smoothing
into the flow.  This is in contrast to, say, a stochastic semilinear heat equation, where the heat kernel would
provide such smoothing.  Indeed, the lack of smoothing is what precludes us from studying space-time white noise forcing in our framework.

After that we turn to the semidiscrete model \eqref{e:finite_kuramoto}, using it as a basis for  constructing a
numerical schemes for the original IVP \eqref{e:kuramoto}.
Theorem~\ref{t:galerkin-converge} establishes that for  \eqref{e:finite_kuramoto}, as
$n\to \infty$, we recover \eqref{e:kuramoto}.   \rev{To obtain rates of convergence, it is necessary to make
additional assumptions on the regularity of the kernel $K(x,y)$ and the process $W$.}
Following \cite{KVMed20}, we use generalized Lipschitz spaces to measure the regularity of $K$
and use the spectral properties of $\bQ$ to classify the
regularity of $W$ to arrive at a rate of convergence, with respect to $n$, in Theorem~\ref{t:rate},
which appears in Section \ref{s:rate}.

Section \ref{s:fulldiscrete} contains our last analytical results. They concern the convergence of the fully discretized problem,
in both space and time, where Euler-Maruyama time stepping is used.  The key results appear
in Theorem \ref{t:fulldiscrete}, and an improved estimate is given in Theorem \ref{t:fulldiscrete2} for a key case
of \eqref{e:kuramoto}, with a trigonometric function for $S$.  Both results establish strong, mean square, convergence.
The fully discrete problem is addressed by splitting the error into a contribution from the spatial discretization of
the associated time discretized problem and the contribution to the error due to the time step in the spatially
continuous problem.  The analysis of the spatial error in the time discretized problem is a natural extension of Theorem~\ref{t:rate}.  A classical analysis of Euler-Maruyama applies, but more effort is needed to obtain the higher order convergence;  again, there is a novel analytical challenge due to the lack of smoothing in the model.

We verify the sharpness of our convergence results with numerical experiments in Section \ref{s:numerics}.
There, we run ensembles of independent trials for different values of particle number $n$ and time step $\Delta t$ and
confirm the predicted scalings in $n$ at fixed $\Delta t$ and vice versa.  These experiments also highlight the
transition between when the error is \rev{dominated by the} deterministic terms and when it is dominated by the
stochastic term.

We conclude with a discussion in Section \ref{s:discussion}, reviewing our results and highlighting open challenges.
Additional computations are given in Appendix \ref{s:supplement}.

\subsection{Related work}
This work is related to two lines of research. On the one hand, there has been a
recent effort in developing numerical methods for nonlocal diffusion  equations
\cite{Du2019b, NochOtaSal16, Noch16, Noch19, BBPN18}. Our contribution to this research is that
first, we consider a stochastically forced problem; second, we work with kernels that may not have much
more regularity beyond integrability; and finally, our model has nonlinear diffusivity and, in this respect,
is a somewhat more general than a typical nonlocal diffusion equation. On the other hand, the systems of SODEs
like \eqref{e:finite_kuramoto}, \eqref{e:finite_kuramoto-ic} may be viewed as interacting diffusions
  on graphs \cite{Lucon2020}. A common framework for modeling interacting diffusions is based on the
  nonlinear process introduced by Sznitman \cite{Sznitman91}. The evolution of each particle is described
  by an implicit nonlinear diffusion equation, which in addition to the unknown state variable involves its probability
  law (see, e.g., \cite{Lucon2020}). In practice, integrating such systems also requires integrating a McKean-Vlasov PDE
  in addition to the system of SODEs for individual particles.
  Our semidiscrete model \eqref{e:finite_kuramoto}, \eqref{e:finite_kuramoto-ic} provides an alternative
  continuum model of interacting diffusions on graphs. A central question in the theory of interacting diffusions
  is analytical description of the continuum (thermodynamic) limit for the system as the number of particles
  tends to infinity. Theorems~\ref{thm:kuramoto} and \ref{t:rate} justify the nonlocal model
  \eqref{e:kuramoto}, \eqref{e:kuramoto-ic} as a continuum limit for \eqref{e:finite_kuramoto},
  \eqref{e:finite_kuramoto-ic} \rev{in the same way as \cite[Theorem~3.1]{Med19} justifies the continuum limit for models without diffusion}.

\section{The initial value problem}
\label{s:IVP}

In this section, we formulate the IVP for the nonlocal diffusion model and study its well-posedness.

\subsection{Preliminaries}
Let $\{\cF_t\subset \cF, t\ge 0\}$ be a normal filtration associated with $W(t)$
\cite{liu2015stochastic}.
Further, let
$T>0$ be arbitrary, but fixed. For $p\ge 2,$ we define $\HpT$, the space of $\cH$-valued
predictable processes\footnote{For the definition of a predictable process and other
  terminology used in  the theory of stochastic integration in Hilbert spaces, we refer
  to \cite{liu2015stochastic}.} $u(t),$ $t\in [0,T],$ such that 
\begin{equation}
    \label{e:Tnorm}
    \trinorm{u}_{p,T} = \esssup_{t\in [0, T]}\trinorm{u(t)}_p,
  \end{equation}
  where $\trinorm{u}_p:=\E[\norm{u(t)}^p]^{1/p}$ and $\norm{\cdot}$ is the norm of $\cH$.  For $\eqref{e:kuramoto}$, it is just $L^2(I^d)$.  $(\HpT, \trinorm{\cdot}_{p,T})$ is a Banach space (cf.~\cite{liu2015stochastic}).

  \subsection{Existence of solutions}
  We first prove existence for  a more general model and then specialize this result to
  \eqref{e:kuramoto}. While the proof is standard (cf. \cite{da2014stochastic, liu2015stochastic}), we include it for completeness.  Consider the equation
\begin{equation}
\label{e:generic}
    du = \bN[t,u] dt +  dW, \quad u(0) = \xi,
\end{equation}
where
$\bN\rev{[t,\bullet]}:\cH\to\cH$ for every $t\in [0,T]$ and $\xi$ is $\cF_0$-measurable random variable.  This problem is posed on an an abstract separable Hilbert space, $\cH$; for \eqref{e:kuramoto} $\cH = L^2(I^d)$.  Further, we assume
\begin{align}
 \label{e:NBound}
     \|\bN[t, u]\|&\leq A_{N} + B_{N}\|u\|\\
\label{e:NLip}
     \|\bN[t,u]-\bN[s, v]\| &\leq  L_{N}  \|u -v\|
\end{align}
for any $u,v\in\cH$ and $t,s\in [0,T]$.

A predictable process $u(t), t\in [0,T],$ is called a mild solution of \eqref{e:generic} if
\begin{equation}
    \label{e:mild}
    u(t) = \xi+ \int_0^t \bN [s,u(s)]ds +  W(t)
  \end{equation}
  holds for all $t\in [0,T]$ a.s. and
  \begin{equation}
    \label{e:mild1}
        \P\paren{\int_0^T \|u(t)\|^2 dt< \infty} =1.
    \end{equation}

\begin{theorem}
    \label{t:existence}
    Let $\xi \in L^{p}(\Omega,\mathcal{F}_0,\P;\cH)$ for some even $p\in \N$.
    Then there exists a unique mild solution to \eqref{e:generic}
    such that
    \begin{equation}
\label{e:databound}
    \trinorm{u}_{p,T}\leq C(1 + \trinorm{\xi}_p),
\end{equation}
where the constant $C$ depends on $T$, but not $\trinorm{\xi}_p$.
\end{theorem}
\begin{proof}[Proof of Theorem \ref{t:existence}]
  Let $\tau=(L_N+1)^{-1}$ and $\xi\in L^p(\Omega, \cF_0, \P; \cH)$,
  $u\in \scH_\tau^p$ and define
\begin{equation}
\label{e:Jfunctional}
\mathbf{J}[u](t) = \xi + \int_0^t \mathbf{N}(s,u(s))ds +  W(t), \quad t\in [0,\tau].
\end{equation}
We want to show that $\bJ$ is a contraction on $\scH_\tau^p$. Since $u(t)$ is a predictable process then so is $\int_0^t \mathbf{N}(s,u(s))ds$ and, consequently, $\mathbf{J}[u](t)$ is predictable too. By the triangle inequality and
\eqref{e:NBound}, for $t\in [0,\tau]$, we have
\begin{equation}\label{bound-J}
\begin{split}
  \trinorm{\bJ [u](t)}_p &\leq \trinorm{\xi}_p +
  \int_0^t \trinorm{\bN(s,u(s))}_p ds  + \trinorm{W(t)}_p\\
  & \leq \trinorm{\xi}_p + \int_0^t \rev{(A_{N} + B_{N} \trinorm{u(s)}_p)}ds +
  \sup_{t\in [0,\tau]}\trinorm{W(t)}_p \\
  &\leq \trinorm{\xi}_p  + \tau \rev{(A_{N} + B_{N} \trinorm{u}_{p,\tau})} +
  \sup_{t\in [0,\tau]}\trinorm{W(t)}_p.
\end{split}
\end{equation}
Since $W(t)$ is a Gaussian process with covariance operator $t \bQ$, we further have
(cf.~\cite{da2014stochastic}) 
\begin{equation}\label{moment}
    \sup_{t\in [0,\tau]}\trinorm{W(t)}_p \leq c_p \sqrt{\tau \Tr\bQ}.
  \end{equation}
  for some $c_p>0$.  The combination of \eqref{bound-J} and \eqref{moment} yields
\begin{equation}\label{preGron}
  \trinorm{\mathbf{J}[u]}_{p,\tau} \leq \trinorm{\xi}_p  +
  \tau \rev{(A_{N} + B_{N} \trinorm{u}_{p,\tau})} + c_p \sqrt{\tau \Tr \bQ}< \infty.
\end{equation}

Next, we demonstrate that $\mathbf{J}$ is a contraction:
\begin{equation*}
    \begin{split}
      \trinorm{\mathbf{J}[u](t) - \mathbf{J}[v](t)}_{p,\tau}&\leq
      \int_0^\tau \trinorm{\bN(s,u(s)) - \bN(s,v(s)) }_p\\
        &\leq \int_0^\tau L_{N}\trinorm{u(s) - v(s)}_pds\\
        & \leq \tau L_{N} \trinorm{u - v}_{p,\tau}.
    \end{split}
\end{equation*}
On account of our choice of $\tau,$ by the Banach contraction mapping principle, $\bJ$ has a unique fixed point in $\scH_\tau^p$.
This yields a unique mild solution of the initial value problem \eqref{e:generic} on
$[0,\tau]$. Using $u(\tau)$ as the initial condition, the local solution can be further extended
to $[0, 2\tau]$ and by repeating this argument again and again, it is extended eventually to $[0,T]$.
Thus, we have constructed a unique mild solution in $\scH_T^p$.  Finally, \eqref{e:databound} follows from \eqref{bound-J} and Gronwall's inequality.
\end{proof}

In addition, we immediately have continuous dependence upon the data and continuity in time:
\begin{corollary}
\label{c:contdep}
Under the same assumptions as Theorem \ref{t:existence}, the solution  depends continuously on initial data
  \begin{equation}\label{cont-data}
\vvvert u(t) - u'(t) \vvvert_{p,T}\lesssim \vvvert \xi - \xi' \vvvert_p
\end{equation}
for any $\xi, \xi' \in L^2(\Omega, \cF_0, \P;\cH)$.
\end{corollary}
\begin{proof}
To show \eqref{cont-data} note
\begin{equation*}
\begin{split}
    \vvvert u(t) - u'(t)\vvvert_p&\leq \vvvert \xi - \xi'\vvvert_p + \int_{0}^t \vvvert \mathbf{N}[s,u(s)] - \mathbf{N}[s, u'(s)]\vvvert_pds \\
    &\leq \vvvert \xi - \xi'\vvvert_p + L_{N}\int_{0}^t \vvvert u(s) - u'(s)\vvvert_pds
\end{split}
\end{equation*}
The Gronwall inequality yields \eqref{cont-data}.
\end{proof}

\begin{corollary}
\label{c:cont}
Under the same assumptions as Theorem \ref{t:existence}, the solution is continuous in time for any $p\geq 2$:
  \begin{equation} \label{u-cont}
    \trinorm{u(t) - u(s)}_p\lesssim \sqrt{|t-s|}.
  \end{equation}
\end{corollary}
\begin{proof}
\begin{equation*}
\begin{split}
    \trinorm{u(t) - u(s)}_p&\leq \int_{s}^t \trinorm{\bN (r,u(r))}_pdr + \trinorm{W(t) - W(s)}_p\\
    & \leq \int_s ^ t \rev{(A_{N} + B_{N}\trinorm{u(r)}_p) dr +  c_p \sqrt{ \Tr \bQ  |t-s|}}\\
    & \leq \rev{(A_{N} + B_{N}\trinorm{u}_{p,T}) |t-s| +  c_p\sqrt{\Tr \bQ  |t-s|}}\\
    & \leq \rev{\paren{(A_{N} + B_{N}\trinorm{u}_{p,T})\sqrt{T} + 
      c_p\sqrt{ \Tr\bQ} }\sqrt{|t-s|}}
\end{split}
\end{equation*}
for any $0\leq s \leq t \leq T$.
Consequently,  $\E[\norm{u(t) - u(s)}^p]\lesssim |t-s|^{p/2}$.

\end{proof}

\subsection{Well-Posedness of the nonlocal evolution equation}
We now prove well-posedness of our model, \eqref{e:kuramoto}.
\begin{theorem}\label{thm:kuramoto}
The IVP for \eqref{e:kuramoto} subject to \eqref{e:fBound},
\eqref{e:fLip}, \eqref{e:SBound}, \eqref{e:SLip}, \eqref{e:Kbounds}
  and given initial condition in $L^p(\Omega, \cF_0,\P;\cH)$ for even
  $p\ge 2$ has a unique mild solution. It depends continuously upon the data and is continuous in time,
  as in Corollaries \ref{c:contdep} and \ref{c:cont}.
\end{theorem}
Note that for the existence of the solution to \eqref{e:kuramoto}, we do not require the Lipschitz
continuity with respect to $t$
in \eqref{e:fLip}.  We will require this later for convergence of the time discretized problem.

\begin{proof}
It is sufficient to verify the conditions of Theorem~\ref{t:existence} for 
\begin{equation}\label{e:our_problem}
    \mathbf{N}[t,u] = f(t,u) + \mathbf{K}[u],
\end{equation}
where $\bK:\cH\to\cH$ is defined by
\begin{equation}
\label{e:Kop}
     \bK[u](x) = \int K(x,y) S(u(x), u(y))dy.
 \end{equation}
 We then proceed in the following steps:
\begin{enumerate}[label=\arabic*., leftmargin=\parindent]
\item First, note that  \eqref{e:fBound}, \eqref{e:fLip}, and \eqref{e:SBound},
\eqref{e:SLip} imply
\begin{subequations}
\label{e:fLq}
\begin{align}
\label{e:fLpBound}
    \|f(t,u)\|_{L^q(I^d)}&\leq A_f + B_f \|u\|_{L^q(I^d)}\\
\label{e:fLqLip}    
  \|f(t,u) - f(t, u')\|_{L^q(I^d)}&\leq L_f  \|u-u'\|_{L^q(I^d)}\\
  \label{e:SLpBound}
    \|S(u,v)\|_{L^q(I^d \times I^d)}&\leq A_S + B_S(\|u\|_{L^p(I^d)} + \|v\|_{L^q(I^d)})\\
    \label{e:SLpLip}
    \|S(u,v) - S(u', v')\|_{L^q(I^d \times I^d)}& \leq L_S(\|u-u'\|_{L^q(I^d)} + \|v - v'\|_{L^q(I^d)})
\end{align}
\end{subequations}
for any $u,v \in L^q(\rev{I^d}), q\in [1,\infty]$. 
In addition, if $\rev{B_f} =0$, then for $p\in [1,\infty)$,
$f(t,\cdot):L^q(I^d)\to L^\infty(I^d)$, with
$\|f(t,u)\|_{L^\infty(I^d)}\leq \rev{A_f}$.
Likewise, if $B_S=0$, for $q\in [1,\infty)$, $S(\cdot,\cdot):L^q(I^d)\times L^q(I^d)\to L^\infty(I^d\times I^d)$, with $\|S(u,v)\|_{L^\infty}\leq A_S$.
 
\item Next, we  show
\begin{align}\label{KBound}
  \|\bK[u]\|_{L^2(I^d)}&\leq A_K + B_K\|u\|_{L^2(I^d)},\\
  \label{KLip}
    \|\bK[u]-\bK[v]\|_{L^2(I^d)} & \leq L_K \|u-v\|_{L^2(I^d)}.
\end{align}
for some nonnegative  $A_K$, $B_K$ and $C_K$ and  all $u,v \in L^2(\rev{I^d})$.  If we can obtain these results, we are done.

\item  To this end, note
\begin{equation*}
\begin{split}
    \|\bK[u]\|_{L^2(I^d)}^2 &\leq \int \set{\int\abs{K(x,y)}\abs{S(u(x),u(y))}dy}^2 dx\\
    &\leq \int \set{\int\abs{K(x,y)}(A_S + B_S|u(x)| + B_S|u(y)|)dy}^2 dx\\
    &\leq 3 \iint |K(x,y)|^2 (A_S^2 + B_S^2|u(x)|^2 + B_S |u(y)|^2) dxdy\\
    &\leq 3A_S^2 \|K\|_{L^2(I^d\times I^d)}^2 + 3 B_S^2 \esssup_{x }\int |K(x,y)|^2 dy\|u\|_{L^2(I^d)}^2 \\
    &\quad + 3 B_S^2 \esssup_{y }\int |K(x,y)|^2 dx\|u\|_{L^2(I^d)}^2
\end{split}
\end{equation*}
Since
\begin{equation*}
    \|K\|_{L^2(I^d\times I^d)}^2\leq \esssup_{x }\int |K(x,y)|^2 dy< \infty
\end{equation*}
\eqref{KBound} holds.

\item For \eqref{KLip},
\begin{equation*}
\begin{split}
\|\bK[u] - \bK[v]\|_{L^2(I^d)}^2&\leq L_S^2 \int\set{\int |K(x,y)| (|u(x) - v(x)| + |u(y) - v(y)|dy}^2 dx\\
&\leq 2 L_S^2 \iint |K(x,y)|^2 (|u(x) -v(x)|^2 + |u(y)-v(y)|^2)dxdy\\
&\leq 2 L_S^2( \esssup_{x }\int |K(x,y)|^2 dy+ \esssup_{y }\int |K(x,y)|^2 dx)\|u-v\|_{L^2(I^d)}^2.
\end{split}
\end{equation*}

\end{enumerate}

\end{proof}

\section{Convergence of the Galerkin scheme}
\label{s:convergence}

In this section, we study convergence of the Galerkin scheme in $L^2(\Omega,\mathcal{F}, \P;L^2(I^d))$ with the associated mean square norm, $\trinorm{\bullet}_2$.  We will also make use of the space time norm \eqref{e:Tnorm} in the case $p=2$.  Additionally, we will assume that our interaction function, $S$, is bounded, which is to say $B_S =0$ in \eqref{e:SBound}.

Let $\mathbf{P}_n$ denote an $L^2$--projector  from $\cH$ onto $\cH^n,$ where $\cH^n$ is defined by \eqref{e:Hnsubspace}.  Our main result of this section is:
\begin{theorem}\label{t:galerkin-converge}
  Let $u(t,x)$ stand for the solution of the IVP for \eqref{e:kuramoto} subject to
  the initial condition $u(0,\cdot)=g\in \cH$ and let $u^n(t,x)$ stand
  for the solution of the finite--dimensional problem
  \eqref{e:finite_kuramoto} subject to $u^n(0,\cdot)=\mathbf{P}_ng$.  Also assume that the interaction term $S$ has $B_S=0$ in \eqref{e:SBound}.
Then
\begin{equation}
\label{e:galerkinconv}
\lim_{n\to \infty} \trinorm{u - u^{n}}_{2,T}=0.
\end{equation}
\end{theorem}
For simplicity, we have taken the initial condition to be deterministic.  The proof of the theorem relies on the following two lemmas.
\begin{lemma}\label{l:convergence}
  There is a positive constant $C=C(A_S, L_f, L_S, K_1, T)$ such that
\begin{equation*}
\trinorm{u - u_{n}}_{2,T}\le C\left(
  \norm{(\mathbf{I}-\mathbf{P}_n)  g} + \|(\bI-\bP^{(1)}_n)K\|_{L^2({I^d\times I^d})}
  +\trinorm{(\mathbf{I}-\mathbf{P}_n) W}_{2,T}\right),
\end{equation*}
where $\bP^{(1)}_n$ and $\bP^{(2)}_n$ stand for 
\rev{$L^2$--projectors of $L^2(I^{2d})$ onto $\mathcal{H}^n\otimes\mathcal{H}$
and $\mathcal{H}\otimes\mathcal{H}^n$} respectively, i.e.,
\begin{alignat*}{1}
  \langle \left(\bI-\bP_n^{(1)}\right) K(\cdot, y), \chi_\bi^n\rangle &=0 \quad  \rev{y\in I^d\; \mathrm{a.e.}},\; \bi\in \nd,\\
  \langle  \left(\bI-\bP_n^{(2)}\right) K(x, \cdot), \chi_\bi^n \rangle &=0 \quad \rev{x\in I^d\; \mathrm{a.e.}},\;\bi\in \nd.
\end{alignat*}
\end{lemma}

\begin{lemma}\label{l:convergeW}
\begin{equation*}
\lim_{n\to\infty}\trinorm{(\mathbf{I}-\mathbf{P}_n) W}_{2,T}=0.
\end{equation*}
\end{lemma}

\begin{proof}[Proof of Theorem~\ref{t:galerkin-converge}]
For any $g\in L^2(I^d)$, we have $\lim_{n\to\infty}\|(\mathbf{I}-\mathbf{P}_n)g\|=0$ (cf.~\cite[Proposition~2.6]{Cha17}). Thus, Theorem~\ref{t:galerkin-converge} follows from Lemmas~\ref{l:convergence} and \ref{l:convergeW}.
\end{proof}

\begin{proof}[Proof of Lemma~\ref{l:convergence}]
  Mild solutions of the IVP for \eqref{e:kuramoto} and \eqref{e:finite_kuramoto} satisfy
  \begin{align}\label{e:mildKM}
    u(t,x) &= g(x) +\int_0^t \{ f(s, u(s,x))+ \mathbf{K} [u(s,\cdot)](x)\} ds + W(t,x)\\
    \label{e:mild-finiteKM}
 u^n(t,x) &= g^n(x) +\int_0^t \{ f(s, u^n(s,x))+ \mathbf{K}^n [u^n(s,\cdot)](x)\} ds + W^n(t,x),
  \end{align}
  where
  \begin{equation}\label{K-n}
    \mathbf{K}^n [v(\cdot)](x)=\int K^n(x,y) S\left(v(x), v(y)\right)dy,
  \end{equation}
  \rev{and $K^n = \mathbf{P}_n K$ is the $L^2(I^d\times I^d)$ projection of $K$ with coefficients as in \eqref{e:Knij}.}
  Our proof then proceeds with the following steps.
  \begin{enumerate}[label=\arabic*., leftmargin=\parindent]

\item Subtracting \eqref{e:mild-finiteKM} from \eqref{e:mildKM} and using the triangle inequality, 
\begin{equation}\label{e:triangle}
  \begin{split}
     \Delta^n(t) & :=\trinorm{ u(t,\cdot)-u^n(t,\cdot)}_2\\
    & \le \|g-g^n\| +\int_0^t\left(\trinorm{f(s, u(s,\cdot))-f(s, u^n(s,\cdot))}_2 \right.\\
    &\left. \quad +\trinorm{\mathbf{K} [u(s,\cdot)]-\mathbf{K} [u^n(s,\cdot)]}_2 
        +\trinorm{\mathbf{K} [u^n(s,\cdot)]-\mathbf{K}^n [u^n(s,\cdot)]}_2\right) ds\\
        &\quad + \trinorm{ W(t,\cdot)-W^n(t,\cdot)}_2.
    \end{split}
  \end{equation}

\item  By \eqref{e:fLip}, 
  \begin{equation} \label{e:use-fLip}
    \trinorm{f(s, u(s,\cdot))-f(s, u^n(s,\cdot))}_2 \le L_f \Delta^n(s).
    \end{equation}
Using \eqref{e:SLip}, \eqref{e:Kbounds} and Jensen inequality, we have
\begin{equation}
\label{e:nonlocalerr}
  \begin{split}
    &\trinorm{\mathbf{K} [u(s,\cdot)]-\mathbf{K} [u^n(s,\cdot)]}_2^2\\
    &= \E\bracket{\int \left( \int K(x,y)\left\{ S(u(s,x), u(s,y))-
        S(u^n(s,x), u^n(s,y)) \right\}dy\right)^2dx}\\
    &\le  L^2_S\int \int K(x,y)^2\E\left\{ |u(s,x)-u^n(s,x)|+|u(s,y)- u^n(s,y)| \right\}^2dxdy\\
    &\le   2L^2_S\rev{(K_1+K_2)} \trinorm{u(s,\cdot)-u^n(s,\cdot)}_2^2.
  \end{split}
  \end{equation}
  The constants, $K_i$, were defined in \eqref{e:Kbounds}.  Thus,
  \begin{equation}\label{e:Kterm}
  \trinorm{\mathbf{K} [u(s,\cdot)]-\mathbf{K} [u^n(s,\cdot)]}_2  \le\rev{\sqrt{2(K_1+K_2)}}L_S\Delta^n(s).
\end{equation}

\item We next need the following observation. If $\phi\in\cH^n$ then
$$
S\left(\phi(x), \phi(\cdot)\right)\in \cH^n\quad \forall x\in I^d
$$
and
\begin{alignat*}{1}
  \mathbf{K}^n [\phi(\cdot)](x) &= \int (\bP_n K)(x,y) S\left(\phi(x),\phi(y)\right)dy\\
  &= \int (\bP^{(2)}_n \bP^{(1)}_n K)(x,y) S\left(\phi(x),\phi(y)\right)dy\\
  &= \int (\bP^{(1)}_n K)(x,y) S\left(\phi(x),\phi(y)\right)dy.
  \end{alignat*}
  In particular,
  \begin{equation}\label{new-observation}
    \mathbf{K}^n [u^n(t,\cdot)](x)=  \int (\bP^{(1)}_n K)(x,y) S\left(u^n(t, x), u^n(t, y)\right)dy.
    \end{equation}

    \item Using \eqref{new-observation} and $|S|\le A_S$,
    we have
\begin{equation}\label{K-Kn}
  \begin{split}
    \trinorm{\mathbf{K} [u^n(s,\cdot)]-\mathbf{K}^n [u^n(s,\cdot)]}_2 
    & \le A_S \|(I-\bP^{(1)}_n) K\|_{L^2(I^d\times I^d)}.
    \end{split}
  \end{equation}
 
    \item Plugging \eqref{e:use-fLip}, \eqref{e:Kterm}, and \eqref{K-Kn}
    into \eqref{e:triangle} and using Gronwall's inequality,
    we obtain
  \begin{equation*}
  \begin{split}
   \sup_{t\in [0,T]} \Delta^n(T) &\le e^{(\sqrt{2}L_SK_1+L_f)T}
   \left(\norm{(\mathbf{I}-\mathbf{P}_n)  g}\right.\\
   &\quad \left.+
     TA_S\|{(\bI-\mathbf{P}^{(1)}_n) \rev{K}}\|_{L^2(I^d\times I^d)} +
     \rev{\vvvert{(\mathbf{I}-\mathbf{P}_n) W}\vvvert_{2,T}}\right).
  \end{split}
    \end{equation*}
    \end{enumerate}
\end{proof}

\begin{proof}[Proof of Lemma~\ref{l:convergeW}]
      We begin \rev{by} calculating 
\begin{equation*}
\begin{split}
  \rev{\trinorm{(\mathbf{I} - \mathbf{P}_n)W(t)}_2^2}& =\E[\|W(t,\cdot)\|^2] - 2 \E[\inner{W(t,\cdot)}{\mathbf{P}_{n} W(t,\cdot)}]+
  \E[\|\mathbf{P}_{n} W(t,\cdot)\|^2]\\
  &= \E[\|W(t,\cdot)\|^2] - \E[\|\mathbf{P}_{n}W(t,\cdot)\|^2]\\
  &=  t \big(\Tr \bQ - \sum_{{\bi}\in[n]^d} h^{-d}\inner{ \bQ\chi^n_\bi}{\chi^n_\bi}\big).
\end{split}
\end{equation*}
Denote the error term
$$
\Delta^n:=\Tr \bQ - \sum_{{\bi}\in[n]^d} h^{-d}\inner{ \bQ\chi^n_\bi}{\chi^n_\bi}.
$$

\begin{enumerate}[label=\arabic*., leftmargin=\parindent]
\item  Expanding the $\chi^n_\bi$ functions in terms of the eigenfunctions of $\bQ$,
$$
\sum_{{\bi}\in[n]^d} \inner{ \bQ\chi^n_\bi}{\chi^n_\bi}=\sum_{k=1}^\infty \inner{\bQ e_k}{e_k}
\sum_{{\bi}\in[n]^d} \inner{\chi^n_\bi}{e_k}^2,
$$
so
\begin{equation}\label{e:def-Delta}
  \Delta^n=\sum_{k=1}^\infty \lambda_k \big(1-\sum_{{\bi}\in[n]^d} h^{-d}\inner{\chi^n_\bi}{e_k}^2\big) = \sum_{k=1}^\infty \lambda_k \|\bP_n^{\perp} e_k\|^2.
\end{equation}
As the projection operator is orthogonal and the $e_k$ are orthonormal, $\|\bP_n^{\perp} e_k\|\leq 1$.

\item Next, let $\epsilon>0$ be arbitrary but fixed.  Since $\bQ$ is trace class, there is $m=m(\epsilon)\in\N$ such that
  \begin{equation}\label{e:remainder}
    0\le \sum_{k=m+1}^\infty\lambda_k<{\frac{\epsilon}{2}}.
  \end{equation}
  Therefore,
  \begin{equation}\label{e:two-sums}
    \Delta^n\le {\frac{\epsilon}{2}} + \Tr\bQ
    \max_{k\in [m]}\|\bP_n^{\perp} e_k\|^2.
  \end{equation}
  
  \item As $n\to \infty$, we are assured that $ \|\bP_n^{\perp} e_k\|\to 0$ (cf.~\cite[Proposition~2.6]{Cha17}). Choosing $n_1=n_1(\epsilon, m)\in\N$ large enough,
  we have, that for all $k\leq m$ and $n\geq n_1$
  \begin{equation}\label{e:second-piece}
  \|\bP_n^{\perp} e_k\|\le {\frac{\epsilon}{\Tr\bQ}}
  \end{equation}
  The combination of \eqref{e:two-sums} and \eqref{e:second-piece} proves that $\Delta^n \to 0$.
  \end{enumerate}
    \end{proof}

\section{The rate of convergence}
\label{s:rate}

To quantify the rate of convergence in Theorem~\ref{t:galerkin-converge}, we need to
impose additional regularity assumptions on the initial data, the kernel $\bK$, and the covariance
operator $\bQ$.
The regularity is well described by Lipschitz spaces, which we define following
\cite{KVMed20}.

\begin{definition}\label{d:Lp-modulus}
For $\phi\in L^p(\rev{I^d}), \,  p\ge 1,$ 
\begin{equation}\label{e:Lp-modulus}
\begin{split}
  \omega_{p}(\phi,\delta&)=\sup_{|h|\le \delta}
  \|\phi(\bullet+h)-\phi(\bullet)\|_{L^p(\rev{I^d_h\bigcap I^d}) },\; \delta>0,\\
  &\quad \rev{I^d_h=\{x\in\R^d: \; x+h\in I^d\},}
\end{split}
\end{equation}
is called the $L^p$-modulus of continuity of $\phi$.
For $\alpha\in (0,1],$ the  Lipschitz space  $\operatorname{Lip}\left(\alpha, L^p(\rev{I^d})\right)$
is defined as follows
\begin{equation}\label{e:Lip-space}
\begin{split}
\operatorname{Lip}\left(\alpha, L^p(\rev{I^d})\right)&=\left\{ \phi\in L^p(\rev{I^d}):\; \exists C>0 :\;
  \omega_p(\phi,\delta)\le C\delta^\alpha\right\},\\
\|\phi\|_{p,\alpha}&=\limsup_{\delta\to 0} \delta^{-\alpha}\omega_p(\phi,\delta).
\end{split}
\end{equation}
\end{definition}

We are now ready to state the main result of this section.
\begin{theorem}\label{t:rate}
  In addition to the assumptions of Theorem \ref{t:galerkin-converge}, let $\lambda_k, k\in\N$ be the eigenvalues of $\mathbf{Q}$ arranged
  in the decreasing order counting multiplicity and  $e_k\in
  L^2(I^d)$ be the corresponding normalized eigenfunctions.
  Let $g\in \operatorname{Lip}\left(\alpha, L^2(I^d)\right)$
and $K\in \operatorname{Lip}\left(\beta, L^2(\rev{I^{d}\times I^d})\right)$
 for some
  $\alpha, \beta\in (0,1]$. 
  Then 
  \begin{equation}\label{e:Galerkin-rate}
    \trinorm{u-u^n}_{2,T}\le C
\max\left\{ n^{-\alpha}, n^{-\beta}, \Psi(n) \right\},
  \end{equation}
  where \rev{
  \begin{equation}
\label{e:Psin}
    \Psi(n) = \Psi(n; \bQ) = \sqrt{\inf_{m\in\N}\rev{\left\{ 
\sum_{k=1}^m \lambda_k\omega_2(e_k, n^{-1})^2 + \sum_{k=m+1}^\infty\lambda_k\right\}}},
\end{equation}
where the eigenvalues $\lambda_k$ and eigenfunctions $e_k$ are those of $\bQ$} and $C>0$ is independent of $n$.
\end{theorem}

The proof of Theorem~\ref{t:rate} relies on the following lemma.
\begin{lemma}[cf.~\cite{KVMed20}]
\label{l:LipBound1} 
  Let $\phi\in L^p(\rev{I^d})$, $p\ge 1$\rev{, and let $\phi_n =\bP_n \phi$.} Then
\begin{equation*}
\|\phi-\phi_n\|_{L^p(I^d)}\le C \omega_p(\phi,\sqrt{d}n^{-1}),
\end{equation*}
where $C$ depends on $d$ but \rev{not} on $\phi$ or $n$.

In particular, if $\phi \in \Lip\left(\alpha, L^2(I^d)\right)$,
$\alpha\in (0,1]$,
\begin{equation}\label{Lipschitz}
\|\phi-\phi_n\|_{L^p(I^d)}\le C n^{-\alpha}.
\end{equation}
\end{lemma}
\begin{remark}\label{r:Holder} 
 Equation~\ref{Lipschitz} with $\alpha=1$ yields the convergence rate for Lipschitz continuous functions.
\end{remark}
\begin{proof}[Proof of Lemma~\ref{l:LipBound1}]
We include a short proof adapted from \cite[Theorem~5]{HOH2010}. 
Using Jensen's inequality and Fubini's theorem, we have
\begin{equation*}
\label{Jensen}
\begin{split}
\|\phi-\phi_n\|_{L^p(I^d)}^p &= \sum_{\bi\in [n]^d} \int_{I^n_\bi} \left| n^{d} \int_{I^n_\bi} \left(\phi(x)-\phi(z)\right)
dz\right|^p dx\\
&\le n^d \sum_{\bi\in [n]^d} \int_{{I^n_\bi}} \int_{I^n_\bi} \left| \phi(x)-\phi(z) \right|^pdzdx\\
&\le \rev{n^{d} } \sum_{\bi\in [n]^d} \int_{I^n_\bi} \int_{B_{\sqrt{d}n^{-1}}:=\{|y|\le \sqrt{d}n^{-1}\} } 
\left| \phi(x)-\phi(x+y)\right|^p    1_{I^d}(x+y) dy dx\\
&= n^{d} \int_{B_{\sqrt{d}n^{-1}}} \int_{\rev{I^d}} \left| \phi(x)-\phi(x+y)\right|^p  1_{I^d}(x+y) dx dy\\
&\le \omega_p^p(\phi, \sqrt{d}n^{-1}) |B_{\sqrt{d}n^{-1}}|n^{d}\\
&=C\omega_p^p(\phi, \sqrt{d}n^{-1}),\quad C=C(d):=|B_{\sqrt{d}n^{-1}}|n^{d}=
{\frac{(\pi d)^{d/2}}{\Gamma\left({\frac{d} {2}}+1\right)}}.
\end{split}
\end{equation*}
where $|B_{\sqrt{d}n^{-1}}|$ stands for the volume of the hypersphere $B_{\sqrt{d}n^{-1}}$.
\end{proof}

\begin{example}\label{ex:Rate}
  Let $\mathbf{Q}=(-\Delta)^{-1}$ and $d=1$.  Then $\lambda_k=(\pi k)^{-2}$ and
  $e_k=\sqrt{2}\sin(\pi k x)$.
  \rev{By a direct application of the mean value theorem},
  \begin{equation*}
    \begin{split}
      &\int_0^1\left(\sin\left(\pi k (x+h)\right)-\sin\left(\pi k x\right) \right)^2 dx\\
      &\rev{=
     \int_0^1 \left(\cos(z_\star(x)) \pi k h\right)^2 dx\leq (\pi k h)^2}
  \end{split}
  \end{equation*}
      Thus,
      \rev{$\omega_2(e_k, h)\leq \pi k h$.}
      Consequently, by optimizing over $m$,  $\Psi(n) = \bigo(n^{-1/2})$.
\end{example}

  \begin{proof}[Proof of Theorem~\ref{t:rate}]
By Lemma \ref{l:convergence},
$$
\trinorm{u - u_n}_{\rev{2,T}}\le C\max\left\{
  \norm{(\mathbf{I}-\mathbf{P}_n)  g}, \|\rev{(\mathbf{I}-\bP^{(1)}_n) {K}}\|_{L^2(I^{d}\times I^{d})},
  \trinorm{(\mathbf{I}-\mathbf{P}_n) W}_{\rev{2,T}} \right\}.
$$
\rev{First, by Lemma~\ref{l:LipBound1}, $\norm{(\mathbf{I}-\mathbf{P}_n)  g}\lesssim n^{-\alpha}$.} 
\rev{Next, since we can write $\bP_n =\bP^{(2)}\bP^{(1)}$, where the projectors are over $L^2(I^d\times I^d)$ and $\bP_n$ is the projector in both $x$ and $y$, 
\begin{equation*}
    \|\bP_nK\|_{L^2(I^d\times I^d)}\leq \|\bP_n^{(1)}K\|_{L^2(I^d\times I^d)}.
\end{equation*}
Next, note that
\begin{equation*}
\begin{split}
    \|(I-\bP_n^{(1)})K\|_{L^2(I^d\times I^d)}^2 &= 
    \|K\|_{L^2(I^d\times I^d)}^2 -\|\bP_n^{(1)} K\|_{L^2(I^d\times I^d)}\\
    &\leq \|K\|_{L^2(I^d\times I^d)}^2 - \|\bP_n K\|_{L^2(I^d\times I^d)}^2
     = \|(I-\bP_n) K\|_{L^2(I^d\times I^d)}^2.
\end{split}
\end{equation*}
Consequently, we can apply  Lemma~\ref{l:LipBound1} again, now over $L^2(I^d \times I^d) = L^2(I^{2d})$ , to conclude
 $\|(\mathbf{I}-\bP^{(1)}_n)K\|_{L^2(I^{d}\times I^{d})}\lesssim n^{-\beta}$. 
}

\rev{It remains}
to estimate $\trinorm{(\mathbf{I}-\mathbf{P}_n) W}_T$. From the proof of Theorem~\ref{t:galerkin-converge},
it follows that
$$
\trinorm{(\mathbf{I}-\mathbf{P}_n) W}^2_{\rev{2,T}}\le T
\sum_{k=1}^\infty \lambda_k  \|\bP_n^{\perp} e_k\|^2
=:\Sigma.
\quad (\mbox{cf. \eqref{e:def-Delta}}).
$$
As in the proof of Theorem~\ref{t:galerkin-converge}, we decompose the sum above into
two contributions:
\begin{equation}\label{e:split}
\Sigma=\underbrace{\sum_{k=1}^m \lambda_k \|\bP_n^{\perp} e_k\|^2}_{\equiv\Sigma_m} +\underbrace{\sum_{k=m+1}^\infty\lambda_k \|\bP_n^{\perp} e_k\|^2}_{\equiv\Sigma_{\bar m}}
\end{equation}
where $m\in\N$ is to be determined.  Again, since the $e_k$ are orthonormal and $\bP_n^{\perp}$ is an orthogonal projector,
\begin{equation}\label{e:2ndsum}
  \Sigma_{\bar m}\le \sum_{k=m+1}^\infty \lambda_k <\Tr \bQ< \infty.
\end{equation}
On the other hand, using Lemma \ref{l:LipBound1}
\begin{equation} \label{1st-sum-a}
  \begin{split}
    \Sigma_m  \le m  \max_{k\in [m]} \lambda_k \|\bP_n^{\perp} e_k\|^2
                       & \le C m \max_{k\in [m]}\omega_2(e_k,\sqrt{d}n^{-1})^2.
                       \end{split}
\end{equation}
The combination of \eqref{e:split}, \eqref{e:2ndsum} and \eqref{1st-sum-a} completes the proof.  
\end{proof}

\section{Fully discrete analysis}
\label{s:fulldiscrete}

Convergence of the semidiscrete problem is interesting in its own right, as we may be interested in
the relationship between a discrete system of particles and its continuum limit 
\rev{(cf.~\cite{Med19})}.
For numerical integration of \eqref{e:kuramoto} in practice,
we must introduce a temporal discretization.  In this section, we analyze that
contribution to the error.

The full discretization of \eqref{e:kuramoto} with Euler-Maruyama time stepping is
\begin{subequations}
\label{e:kuramotodiscrete}
\begin{align}
    u^{n,k+1} &= u^{n,k} + f(t_k,u^{n,k})\Dt + \mathbf{K}^n[u^{n,k}]\Dt +  \Delta W^{n,k+1},\\
    u^{n,0} &= \mathbf{P}^n g
\end{align}
\end{subequations}
where $u^{n,k}$ is our approximation of \rev{the solution in the Galerkin space $\mathcal{H}^n$} at time $t_k$.   $\mathbf{K}^n$ is defined as in \eqref{K-n}, and
\begin{equation}
\label{e:Winc}
    \Delta W^{n,k+1} = \mathbf{P}^n(W(t_{k+1}) - W(t_k)) = W^{n,k+1} - W^{n,k}
\end{equation}
is the increment in the Gaussian process within the subspace.
Iterating, 
\begin{equation}
\label{e:kuramaotodiscretesum}
\begin{split}
    u^{n,k} &= u^{n,0}  + \sum_{j=0}^{k-1}\Delta t  f(t,u^{n,j}) + \sum_{j=0}^{k-1}\Delta t \mathbf{K}^n[u^{n,j}] + \sum_{j=0}^{k-1}  \Delta W^{n, j},\\
    &=u^{n,0}  + \sum_{j=0}^{k-1}\Delta t  f(t,u^{n,j}) + \sum_{j=0}^{k-1}\Delta t \mathbf{K}^n[u^{n,j}]+  W^{n,k}.
\end{split}
\end{equation}

Our goal is to obtain a convergence rate, with respect to both $n$, the spatial mesh, and $\Delta t$, the time step, for the error
\begin{equation}
\label{e:error}
    \Delta^{n,k} =\vvvert u(t_k) - u^{n,k}\vvvert_2
\end{equation}
along with the max error,
\begin{equation}
    \label{e:errornorm}
    \max_{k\leq M} \Delta^{n,k}
\end{equation}
We will assume that the time steps are chosen such that
\begin{equation}
    M = \frac{T}{\Dt} \in \mathbb{N}.
\end{equation}
Throughout, $n$ will be used to denote spatial discretization, while $j$ and $k$, will indicate the associated time, $t_j = j \Dt$.    As we noted after stating Theorem \ref{thm:kuramoto}, we will now make use of the Lipschitz continuity with respect to $t$ in assumption \eqref{e:fLip}.  

To better analyze time and spatial discretization error, we break the problem of estimating \eqref{e:errornorm} into two intermediate problems, one addressing only spatial error and another addressing only time error:
\begin{equation}
\label{e:error2}
    \Delta^{n,k} \leq \underbrace{\vvvert u(t_k) - u^k\vvvert_2}_{\equiv \Delta_t^k} + \underbrace{\vvvert u^k - u^{n,k} \vvvert_2}_{\equiv \Delta_x^{n,k}}.
\end{equation}
The term $\Delta_t^k$ accounts for only time discretization and $\Delta_x^{n,k}$ accounts for space discretization.  The time step $\Dt$ is still present in $\Delta_x^{n,k}$, but the error with respect to $n$ is uniform over $\Dt \in (0, \Dt_0)$ for any fixed $\Dt_0> 0$.  Decomposition \eqref{e:error2} introduces a new quantity, $u^k$, which corresponds to the discretization of \eqref{e:kuramoto} only in time,
\begin{equation}
\label{e:Dtkuramoto}
\begin{split}
    u^{k+1} &= u^k + f(t_k, u^k)\Dt + \bK[u^k]\Dt + \Delta W^{k+1}\\
    u^0 & = g
\end{split}
\end{equation}

For analysis, it is helpful to represent the exact solution as
\begin{equation}
\label{e:exactsum}
    u(t_k) = g + \sum_{j=0}^{k-1}\int_{t_{j}}^{t_{j+1}}f(s,u(s)ds  + \sum_{j=0}^{k-1}\int_{t_{j}}^{t_{j+1}}\mathbf{K}[u(s)] ds +  W(t_k),
\end{equation}
along with its time discretization, \eqref{e:Dtkuramoto},
\begin{equation}
    \label{e:Dtsum}
    u^{k}=u^{0}  + \sum_{j=0}^{k-1}\Delta t  f(t,u^{j}) + \sum_{j=0}^{k-1}\Delta t \mathbf{K}[u^{j}]+  \rev{W^{k}}.
\end{equation}

The main results of this section is:
\begin{theorem}
\label{t:fulldiscrete}
Under the same assumptions as those of Theorem \ref{t:rate}
\begin{equation*}
    \max_k\vvvert u(t_k) - u^{n,k}\vvvert_2 \lesssim \max\{ n^{-\alpha}, n^{-\beta}, \Psi(n)\} + \sqrt{\Delta t}
\end{equation*}
where $\Psi$ is defined in \eqref{e:Psin}.
\end{theorem}
\begin{proof}
Using \eqref{e:error2} along with Corollary \ref{c:semidiscreteconv} and Corollary \ref{c:EMkuramoto}, we have our result.
\end{proof}
An improvement to this, with $\bigo(\Dt)$ error, for the particular case of \eqref{e:kuramoto}, where $S$ is a trigonometric function, is presented in Section \ref{s:milstein}

\subsection{Spatial error of the discrete in time problem}

As a corollary to Theorems \ref{t:rate}, we have
\begin{corollary}
\label{c:semidiscreteconv}
Under the assumptions of  Theorems \ref{t:rate}, fixing $\Dt_0 >0$, for all $\Dt \in (0, \Dt_0)$,
\begin{equation*}
    \max_{k}\vvvert u^k - u^{n,k}\vvvert_2\lesssim \max\{ n^{-\alpha}, n^{-\beta}, \Psi(n) \}.
\end{equation*}
\end{corollary}
The implicit constant in the above error bound depends upon $\Dt_0$ but not $\Dt$, so the result is uniform for all $\Dt$ sufficiently small.

\begin{proof}
The proof is established by reformulating Lemma \ref{l:convergence} for discrete sums in place of time integrals.  We begin with the computation
\begin{equation*}
\begin{split}
\Delta_x^{n,k} &\equiv \vvvert u^k - u^{n,k}\vvvert_2 \leq \vvvert (\mathbf{I} - \mathbf{P}_n) g \vvvert + \Dt \sum_{j=0}^{k-1} \vvvert f(t_j,u^j) - f(t_j, u^{n,j})\vvvert_2  \\
&\quad + \Dt \sum_{j=0}^{k-1} \vvvert \mathbf{K}[u^j] - \mathbf{K}^n[u^{n,j}]\vvvert_2 + \vvvert (\mathbf{I} - \mathbf{P}_n) W(t_k) \vvvert_2
\end{split}
\end{equation*}
For the self interaction summand, by our assumptions on $f$,
\begin{equation*}
    \vvvert f(t_j,u^j) - f(t_j, u^{n,j})\vvvert_2 \leq L_f \vvvert u^j - u^{n,j} \vvvert_2 = L_f \Delta_x^{n,j}.
\end{equation*}
For the nonlocal summand, in the case that $K\in \Lip(\rev{\beta}, L^{2}(I^{2d}))$,
\begin{equation*}
\begin{split}
    \vvvert \mathbf{K}[u^j] - \mathbf{K}^n[u^{n,j}]\vvvert_2 &\leq L_k \Delta_x^{n,j} + A_S \|(\mathbf{I} - \mathbf{P}_n)K\|\\
    &\lesssim \Delta^{n,j}_x + n^{-\beta}.
\end{split}
\end{equation*}
Finally, as in the proof of Theorem \ref{t:rate}
\begin{equation*}
    \vvvert (\mathbf{I} - \mathbf{P}^n) W(t_k) \vvvert_2\lesssim \Psi(n).
\end{equation*}
\rev{Next, since
\begin{equation*}
     \Delta_{x}^{n,0} = \vvvert u^{0} - u^{n,0} \vvvert_2 = \| (\mathbf{I} - \mathbf{P}_n) g\| \lesssim n^{-\alpha},
\end{equation*}
and
\begin{equation*}
\begin{split}
    \Delta_x^{n,k} &\lesssim n^{-\alpha } + \Dt\sum_{j=0}^{k-1} \Delta_x^{n,j} + \Dt \sum_{j=0}^{k-1} n^{-\beta} +  \Psi(n)\\
    &\lesssim \max \{n^{-\alpha }, n^{-\beta}, \Psi(n)\}  +  \Dt\sum_{j=0}^{k-1} \Delta_x^{n,j},
\end{split}
\end{equation*}
we can apply apply a discrete Gronwall equality to obtain
\begin{equation*}
    \Delta_x^{n,k}\lesssim \max \{n^{-\alpha }, n^{-\beta}, \Psi(n)\}  e^{\Dt k}.
\end{equation*}
This completes the result.
}

\end{proof}

\subsection{Time stepping error}

To unify our analysis of the time stepping error, we return to the generic form \eqref{e:generic}, and compare
\begin{align}
    \label{e:uN1}
    u(t_k) &= g + \sum_{j=0}^{k-1}\int_{t_j}^{t_{j+1}} \mathbf{N}[s, u(s)]ds + W(t_k)\\
    \label{e:uN2}
    u^k & = u^0 + \sum_{j=0}^{k-1} \mathbf{N}[t_j, u^j]\Dt  + W^k
\end{align}
This amounts to the Euler-Maruyama discretization, which is known to have a strong order of convergence of $1/2$.  We will establish a convergence result for \eqref{e:uN2}, and then verify $f$, $K$, and $S$ in \eqref{e:kuramoto} satisfy the assumptions, as in the proof of Theorem \ref{thm:kuramoto}.  In place of \eqref{e:NLip}, we will need the stronger assumption
\begin{equation}
\label{e:NLip2}
    \|\mathbf{N}[t,u] - \mathbf{N}[s,v]\| \leq L_N(|t-s| + \|u-v\|).
\end{equation}

\begin{theorem}
\label{t:EM}
Under the assumptions of Theorem \ref{t:existence} and \eqref{e:NLip2}, \rev{for $\Dt>0$}, the time discretization error satisfies
\begin{equation*}
    \max_{k}\vvvert u(t_k) - u^k \vvvert_2 \lesssim \sqrt{\Dt}
\end{equation*}
\end{theorem}
An immediate consequence of this is the result for \rev{\eqref{e:kuramoto}},
\begin{corollary}
\label{c:EMkuramoto}
Under the assumptions of Theorem \ref{thm:kuramoto}, \rev{for $\Dt >0$}, the time discretization error satisfies
\begin{equation*}
    \max_{k}\vvvert u(t_k) - u^k \vvvert_2 \lesssim \sqrt{\Dt}
\end{equation*}
\end{corollary}

We include a proof of Theorem \ref{t:EM}, which is standard, for completeness.
\begin{proof}[Proof of Theorem \ref{t:EM}]
Letting  $\Delta_t^k = \vvvert u(t_k) - u^k\vvvert_2$, our first estimate is
\begin{equation*}
    \Delta_t^k \leq \sum_{j=0}^{k-1}\int_{t_j}^{t_{j+1}} \vvvert \mathbf{N}[s,u(s)] - \mathbf{N}[t_j, u^j] \vvvert_2 ds
\end{equation*}
By our assumptions and Theorem \ref{thm:kuramoto},
\begin{equation*}
\begin{split}
    \vvvert \mathbf{N}[s,u(s)] - \mathbf{N}[t_j, u^j] \vvvert_2   & \lesssim |s-t_j| + \vvvert u(s) - u^j\vvvert_2\\
    &\lesssim |s-t_j| + \vvvert u(s) - u(t_j)\vvvert_2 + \Delta^j_t\\
    &\lesssim |s-t_j| + \sqrt{|s-t_j|} + \Delta_t^j.
\end{split}
\end{equation*}
Consequently,
\begin{equation*}
\begin{split}
    \Delta_t^k &\lesssim \sum_{j=0}^{k-1} \int_{t_j}^{t_{j+1}} \sqrt{|s-t_j|} + |s-t_j| + \Delta_t^j\\
    &\lesssim \sum_{j=0}^{k-1} \Dt^{3/2} + \Dt^2 + \Dt \Delta_t^j\lesssim \sqrt{\Dt} + \Dt \sum_{j=0}^{k-1}  \Delta_t^j.
\end{split}
\end{equation*}
\rev{Since $u^0 = u(0)$, $\Delta_t^0 = 0$, by discrete Gronwall,
\begin{equation*}
     \Delta_t^k \lesssim \sqrt{\Dt} e^{k\Dt}.
\end{equation*}
This completes the result.  
}
\end{proof}

\subsection{Improved convergence estimates}
\label{s:milstein}

Higher order convergence in time can be achieved in certain special, but important cases.
This is a consequence of our problem having only additive noise and the interaction term in the
classical Kuramoto being a trigonometric function.  For additive noise, Euler-Maruyama is exactly
Milstein's method which has strong first order convergence, provided the drift term is sufficiently smooth,
\cite{lord2014an,kloeden2013numerical}.   We are able to prove:

\begin{theorem}
\label{t:fulldiscrete2}
Under the same assumptions as in Theorem \ref{t:rate}, if, in addition, \rev{$f=0$ and} $S(u,v) = \sin(2\pi(u-v))$, then
\begin{equation*}
    \max_k\vvvert u(t_k) - u^{n,k}\vvvert_2 \lesssim \max\{ n^{-\alpha}, n^{-\beta}, \Psi(n)\} + {\Delta t}
\end{equation*}
where $\Psi$ is defined in \eqref{e:Psin}.
\end{theorem}
\begin{proof}
This follows from \eqref{e:error2}, the previously stated Corollary \ref{c:semidiscreteconv}, and Corollary \ref{c:KuramotoMilstein}, which is presented below.
\end{proof}

This result is rather specialized to the $\sin$ function, though it can be generalized to other such trigonometric functions and their linear combinations.   However, it reveals a fundamental challenge to studying \eqref{e:kuramoto} owing to the lack of smoothing.

For equations with additive noise, to obtain the higher order in time result,  one typically assumes at most linear bounds with respect to $u$ on the first and second variations of $\mathbf{N}[t,\bullet]:\rev{\mathcal{H}\to \mathcal{H}}$, as in \cite{lord2014an,Lord.2018,Wang.2015}.  That is to say, it is assumed 
$$\|D\mathbf{N}[t,u]\|_{\rev{\mathcal{H}\to \mathcal{H}}}\lesssim \|u\|.
$$  
The higher order convergence result is then obtained by performing a Taylor expansion in the nonlinearity, using such assumed bounds on the variational derivatives.  Here, there is an obstacle in even defining the variational derivatives. Consider the case of
\begin{equation*}
   \mathbf{N}[t,u](x) =  \int K(x,y) \sin(u(x) - u(y))dy
\end{equation*}
By Taylor's theorem with remainder, 
\begin{equation*}
    \sin(u + \delta u)= \sin(u) + \cos(u)\delta u -\int_0^1 (1-\lambda) \sin(u + \lambda \delta u) \delta u^2 d\lambda.
\end{equation*}
Consequently,
\begin{equation*}
\begin{split}
    &\mathbf{N}[t,u + \delta u] = \mathbf{N}[t,u] + \int K(x,y) \cos( u(x) - u(y)) (\delta u(x) - \delta u(y))dy\\
    &\quad - \int K(x,y) \int_0^1 (1-\lambda) \sin(u(x) - u(y) + \lambda (\delta u (x) - \delta u(y)))(\delta u(x) - \delta u(y))^2d\lambda
\end{split}
\end{equation*}
To justify that the first variational derivative is
\begin{equation*}
 D\mathbf{N}[u]\delta u =    \int K(x,y) \cos( u(x) - u(y)) (\delta u(x) - \delta u(y))dy
\end{equation*}
we would need to show that the quadratic  term is $\littleo(\|\delta u\|_{L^2}^2)$.  But the second order term in the expansion includes expressions like
\begin{equation*}
    \set{\int K(x,y)\int_0^1 (1-\lambda) \sin(u(x) - u(y) + \lambda (\delta u (x) - \delta u(y))) dy \rev{d\lambda}}(\delta u(x))^2.
\end{equation*}
This necessitates \rev{$\delta u \in L^4(I^d)$, but our solutions, in the spatial variable, are only in $\mathcal{H}=L^2(I^d)$.} Thus, the standard approach, via variational derivatives will not work here.

A sufficient condition on the nonlinearity to obtain the Milstein rate of convergence is the following:
\begin{proposition}
\label{p:Milstein}
Under the same assumptions of Theorem \ref{t:existence} and \eqref{e:NLip2}, assume, also,
that for any partition $0=t_0<t_1<\ldots t_{M}=T$, there exists a constant $C$, uniform over the partition, and $\mathcal{H}$ valued functions $a_j$ and $\beta_j$ such that for $s \in [t_j, t_{j+1}]$
\begin{gather*}
\mathbf{N}[t_j, u(s)] - \mathbf{N}[t_j, u(t_j)] = a_j(s) + \beta_j(s),\\
    \trinorm{a_j(s)}_2 \leq C (s-t_j),\\
    \trinorm{\beta_j(s)}_2 \leq C \sqrt{s-t_j}, \quad \E[\beta_j(s)\mid \mathcal{F}_{t_j}] = 0.
\end{gather*}
\rev{Then for all $\Dt>0$,}
\begin{equation*}
    \max_k \vvvert u(t_k) - u^k\vvvert_2 \lesssim \Dt.
\end{equation*}
\end{proposition}
This  avoids the need to directly manage the problematic variational derivatives of the drift term.

\begin{corollary}
\label{c:KuramotoMilstein}
For \eqref{e:kuramoto}, in the case that $f=0$ and \rev{$S(u,v) = \sin(2\pi(u-v))$ the} assumptions of Theorems \ref{t:existence} and \ref{p:Milstein} are satisfied with $\mathbf{N}[t, u(t)] = \mathbf{K}[u(t)]$.  For this model, we have $\bigo(\Dt)$ convergence under an Euler-Maruyama discretization.
\end{corollary}
\begin{proof}[\rev{Proof of Corollary \ref{c:KuramotoMilstein}}]
This follows from Proposition \ref{p:Milstein}, once the conditions are verified on the nonlinearity. This is a somewhat technical proof which we omit from the main text.  See Proposition \ref{p:MilsteinKuramoto} in the appendix for the full details.
\end{proof}

\begin{proof}[Proof of Proposition \ref{p:Milstein}]
\begin{enumerate}[label=\arabic*., leftmargin=\parindent]
\item As in the case of proof of Theorem \ref{t:EM}, we begin with
\begin{equation*}
    \begin{split}
        \Delta_t^k &= \trinorm{\sum_{j=0}^{k-1} \int_{t_j}^{t_{j+1}}\mathbf{N}[s,u(s)] - \mathbf{N}[t_j, u^j]ds }_2\\
        &\leq \underbrace{\trinorm{\sum_{j=0}^{k-1} \int_{t_j}^{t_{j+1}}\mathbf{N}[s,u(s)] - \mathbf{N}[t_j, u(s)]ds }_2}_{I} + \underbrace{\trinorm{\sum_{j=0}^{k-1}\int_{t_j}^{t_{j+1}} \mathbf{N}[t_j,u(s)] - \mathbf{N}[t_j, u(t_j)]ds }_2}_{II}\\
        &\quad + \underbrace{\trinorm{\sum_{j=0}^{k-1}\int_{t_j}^{t_{j+1}} \mathbf{N}[t_j,u(t_j)] - \mathbf{N}[t_j, u^j]ds }_2}_{III}
    \end{split}
\end{equation*}
Each of the three terms will be treated separately.
\item First, \rev{by \eqref{e:NLip2},}
\begin{equation*}
\begin{split}
     I&\leq \sum_{j=0}^{k-1}\int_{t_j}^{t_{j+1}} \trinorm{\mathbf{N}[s,u(s)] - \mathbf{N}[t_j, u(s)]}_2ds \leq \sum_{j=0}^{k-1}\int_{t_j}^{t_{j+1}} L_N (s-t_j) ds \\
     &\lesssim \sum_{j=0}^{k-1}\Dt^2 \lesssim \Dt
\end{split}
\end{equation*}
\item Next, 
\begin{equation*}
    \begin{split}
        III&\leq \sum_{j=0}^{k-1}\int_{t_j}^{t_{j+1}} \trinorm{\mathbf{N}[t_j,u(t_j)] - \mathbf{N}[t_j, u^j]}_2ds\\
        &\leq \sum_{j=0}^{k-1}\int_{t_j}^{t_{j+1}} L_N\trinorm{u(t_j)-u^j}_2ds \lesssim \Dt \sum_{j=0}^{k-1} \Delta_t^j 
    \end{split}
\end{equation*}
\item Finally,
\begin{equation*}
    \begin{split}
        II&\leq \sum_{j=0}^{k-1}\int_{t_j}^{t_{j+1}} \trinorm{a_j(s)}_2ds  +  \trinorm{\sum_{j=0}^{k-1}\int_{t_j}^{t_{j+1}} \beta_j(s) ds }_2
    \end{split}
\end{equation*}
By our assumptions,
\begin{equation*}
     \sum_{j=0}^{k-1}\int_{t_j}^{t_{j+1}} \trinorm{a_j(s)}_2ds\leq C  \sum_{j=0}^{k-1}\int_{t_j}^{t_{j+1}} (s-t_j)ds \lesssim \Dt
\end{equation*}
For the other term, 
\begin{equation*}
\begin{split}
   \trinorm{\sum_{j=0}^{k-1}\int_{t_j}^{t_{j+1}} \beta_j(s) ds }_2^2 &= \sum_{j=0}^{k-1} \trinorm{\int_{t_j}^{t_{j+1}} \beta_j(s) ds}_2^2 \\
   &\quad + 2 \sum_{i<j}\E\bracket{\inner{\int_{t_i}^{t_{i+1}} \beta_{\rev{i}}(s) ds}{\int_{t_j}^{t_{j+1}} \beta_j(s) ds}}
\end{split}
\end{equation*}
Conditioning on $\mathcal{F}_{t_j}$,
\begin{equation*}
\begin{split}
    &\E\bracket{\inner{\int_{t_i}^{t_{i+1}} \beta_{\rev{i}}(s) ds}{\int_{t_j}^{t_{j+1}} \beta_j(s) ds}}\\
    &= \E\bracket{\E\bracket{\inner{\int_{t_i}^{t_{i+1}} \beta_{\rev{i}}(s) ds}{\int_{t_j}^{t_{j+1}} \beta_j(s) ds}\mid\mathcal{F}_{t_j} }}\\
    &= \E\bracket{\inner{\int_{t_i}^{t_{i+1}} \beta_{\rev{i}}(s) ds}{\int_{t_j}^{t_{j+1}} \E[\beta_j(s)\mid\mathcal{F}_{t_j} ] ds}} = 0
\end{split}
\end{equation*}
For the remaining  terms, applying Jensen and our assumption,
\begin{equation*}
    \begin{split}
    \sum_{j=0}^{k-1} \trinorm{\int_{t_j}^{t_{j+1}} \beta_j(s) ds}_2^2 &\leq \sum_{j=0}^{k-1} \Dt\int_{t_j}^{t_{j+1}}\trinorm{ \beta_j(s) }_2^2ds\\
    &\lesssim \sum_{j=0}^{k-1} \Dt\int_{t_j}^{t_{j+1}} (s-t_j) ds\\
    &\lesssim \sum_{j=0}^{k-1} \Dt^3 \lesssim \Dt^2
    \end{split}
\end{equation*}
\item Combining all of our estimates on $I$, $II$, and $III$,
\begin{equation*}
    \Delta_t^k \lesssim \Dt + \Dt \sum_{j=0}^{k-1} \Delta_t^j.
\end{equation*}
\rev{As $\Delta_t^0 = 0$, by the discrete Gronwall inequality, 
\begin{equation*}
    \Delta_t^k \lesssim \Delta t e^{k \Dt},
\end{equation*}
completing the proof}


\end{enumerate}
\end{proof}

\section{Numerical examples}
\label{s:numerics}

In this section we present numerical experiments to demonstrate our convergence results.  While our time stepping error, from Theorem \ref{c:KuramotoMilstein}, appears to be sharp, there appears to be opportunity to refine the spatial error given in Theorems \ref{c:semidiscreteconv} and \ref{t:fulldiscrete2}.

As a test problem, we consider  the problem in $d=1$
\begin{equation}
    du = \int K(x,y) \sin(2\pi(u(x)-u(y))) dy dt + dW
\end{equation}
and
\begin{subequations}
\label{e:Kexample}
\begin{align}
    A_r & = \set{(x,y) \in [0,1]^2\mid \min\{|x-y|, 1- |x-y|\}<r}\\
    K(x,y)& =  1_{A_r}(x,y)
\end{align}
\end{subequations}

As an initial condition, we take
\begin{equation}
    u_0(x) = x(1-x).
\end{equation}
The stochastic process $W$ has $\mathbf{Q} = (-d^2/dx^2)^{-s/2}$ with periodic boundary conditions.  The parameter $s>1$ ensures that $\bQ$ is trace class on $\mathcal{H}$.

For such a process, since the initial condition is continuous, we can take $\alpha=1$ (where $\alpha$ is given in Theorem \ref{t:rate}).  For a piecewise constant interaction kernel function, $\beta=1/2$ (see \cite{KVMed20}).  Lastly, for $\mathbf{Q}$, since the eigenfunctions $e_k$ are trigonometric functions, as in the case of Example \ref{ex:Rate}, we will have that $\omega_2(e_k, n^{-1}) \lesssim k/n$, and the eigenvalues scale as $\lambda_k \sim k^{-s}$ with $s>1$.  Then, as in Example \ref{ex:Rate} allows us to conclude that
\begin{equation}
\begin{split}
    \Psi(n)& \lesssim \sqrt{\inf_m \set{\rev{\frac{1}{n^2}\sum_{k=1}^m k^{2-s} + \sum_{k=m+1}^\infty k^{-s} }}}\\
    &\lesssim \sqrt{ \inf_m \set{\rev{\frac{1}{n^2}\sum_{k=1}^m k^{2-s} + m^{1-s}}}}
\end{split}
\end{equation}
\rev{For $s>1$ and $s \neq 3$, $\Psi(n)\lesssim n^{-(s-1)/2}$.  For $s=3$, $\Psi(n)\lesssim n^{-1}\sqrt{\log n}$.}  Therefore, looking at the mean square error, Theorem \ref{t:fulldiscrete2} predicts 
\begin{equation}
\label{e:mse}
    \text{MSE}\lesssim n^{-2} + n^{-1}+ \Dt^2 +\rev{\begin{cases}
    n^{-(s-1)} & s>1, \; s\neq 3\\
    n^{-2}\log n & s=3
    \end{cases}}
\end{equation}
At first glance, it would appear that for $s>2$, the contribution to the spatial error is dominated by the contribution from the nonlocal term, $n^{-1}$, while for $s< 2$, the spatial error is dominated by the noise term, $n^{-(s-1)}$.  In fact, our numerical experiments will reveal that the contribution to the MSE from the nonlocal term is actually $O(n^{-2})$, and, instead, the noise term dominates for $s<3$.

\subsection{Results and details of computation}

As we do not have access to an analytic solution, we make use a high resolution solution with $n = n_\star$ large, as a surrogate to see convergence in $n$.  Indeed, at a fixed $\Dt$ by Corollary \ref{c:semidiscreteconv}, since
\begin{equation*}
\begin{split}
    \max_k \vvvert{u^k - u^{n,k}}\vvvert_2&\leq \max_k \vvvert{u^k - u^{n_\star,k}}\vvvert_2 + \max_k \vvvert{u^{n_\star,k} - u^{n,k}}\vvvert_2\\
    &\lesssim  \max_k \vvvert{u^{n_\star,k} - u^{n,k}}\vvvert_2
\end{split}
\end{equation*}
provided we take $n_\star$ large enough.  Analogously, at a fixed $n$, by taking $\Dt_\star$ small enough, we compare against $\Dt$
\begin{equation*}
\begin{split}
    \vvvert u^{n,M} - u^n(T)\vvvert_2&\leq   \vvvert u^{n,M} - u^{n, M_{\star}}\vvvert_2 + \vvvert u^{n, M_{\star}} - u^n(T)\vvvert_2\\
    &\lesssim\vvvert u^{n,M} - u^{n, M_{\star}}\vvvert_2,
\end{split}
\end{equation*}
where $M_{\star} = T/\Dt_\star$.

In each case, we perform $10^2$ independent trials.  To see the convergence in $n$, we fix $\Dt=0.001$ and vary $n$, along with $s$.  To see the convergence in $\Dt$, we fix $n=1024$ and vary $\Dt$, along with $s$.    The random process is sampled by FFT methods.  When assessing the convergence in $\Dt$, it is sampled on $n=1024$ points.  For convergence in $n$ at fixed $\Dt$,  we sample the process on $2^{14}$ mesh points, and project it onto the lower resolution in $n$  spaces by Riemann sum approximation.  As this is higher resolution than the values of $n$ at which we compare, the Riemann approximation error is higher order.  The discretized interaction kernel, ${\bf P}_n K$, is computed using \rev{Gauss-Kronrod} quadrature, and, in assessing the $L^2(I^d)$ error, \rev{Gauss-Kronrod} is also used to compare the piecewise constant approximations at across resolutions.

The spatial results appear in Figure \ref{fig:nconv1}.  For \rev{$s < 3$}, the squared stochastic error, $\propto n^{-(s-1)}$ dominates. For $s>3$, it is   dominated by an error,  $\propto n^{-2}$.  It was predicted that the squared nonlocal discretization error, $\propto n^{-1}$ would dominate for $s \geq 2$.  We explain this discrepancy below, but, briefly, it is due to the square of the nonlocal integral error actually being $\propto n^{-2}$ for this $K(x,y)$.

\begin{figure}
    \centering
    \subfigure[]{\includegraphics[width=6.25cm]{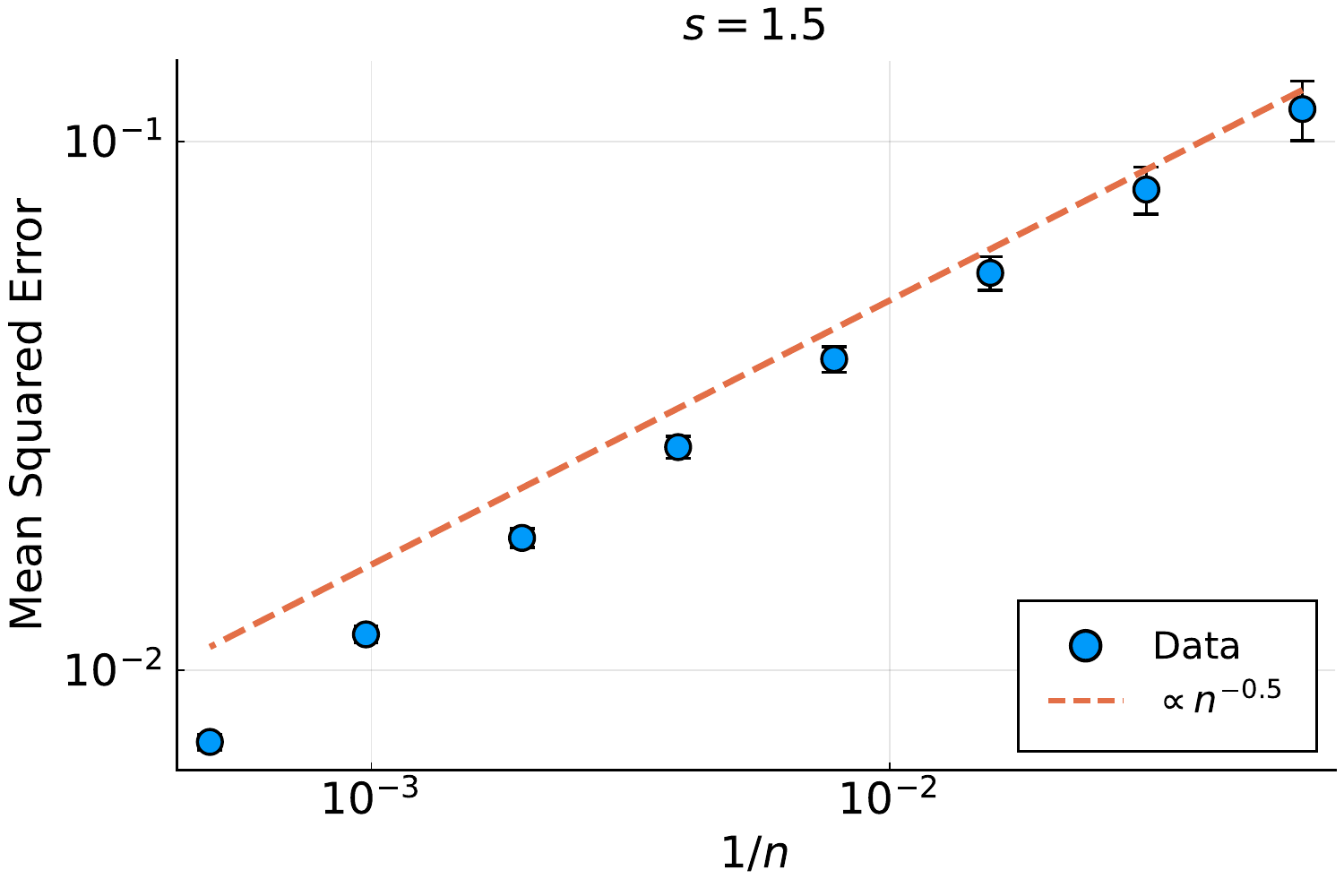}}
    \subfigure[]{\includegraphics[width=6.25cm]{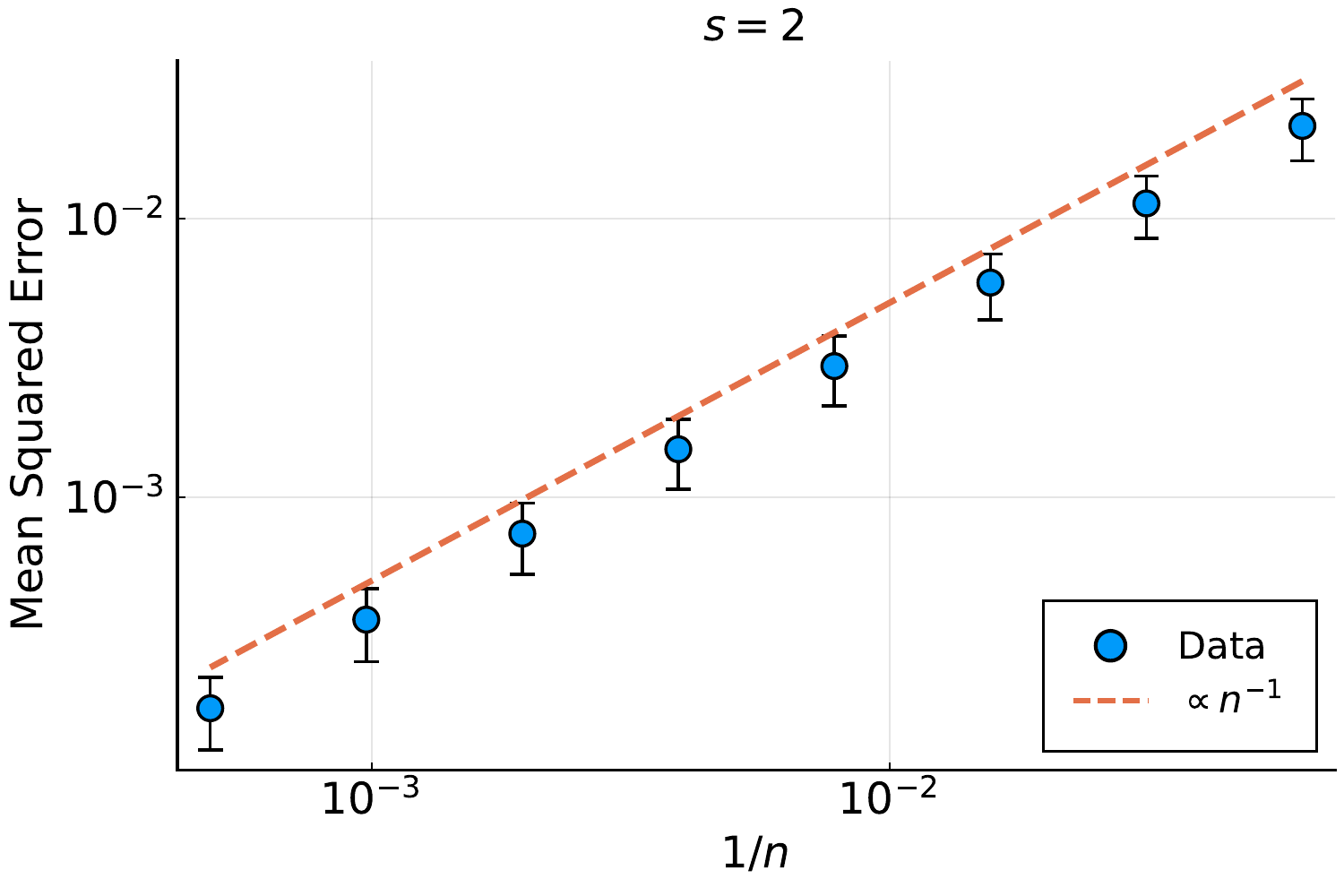}}
        
    \subfigure[]{\includegraphics[width=6.25cm]{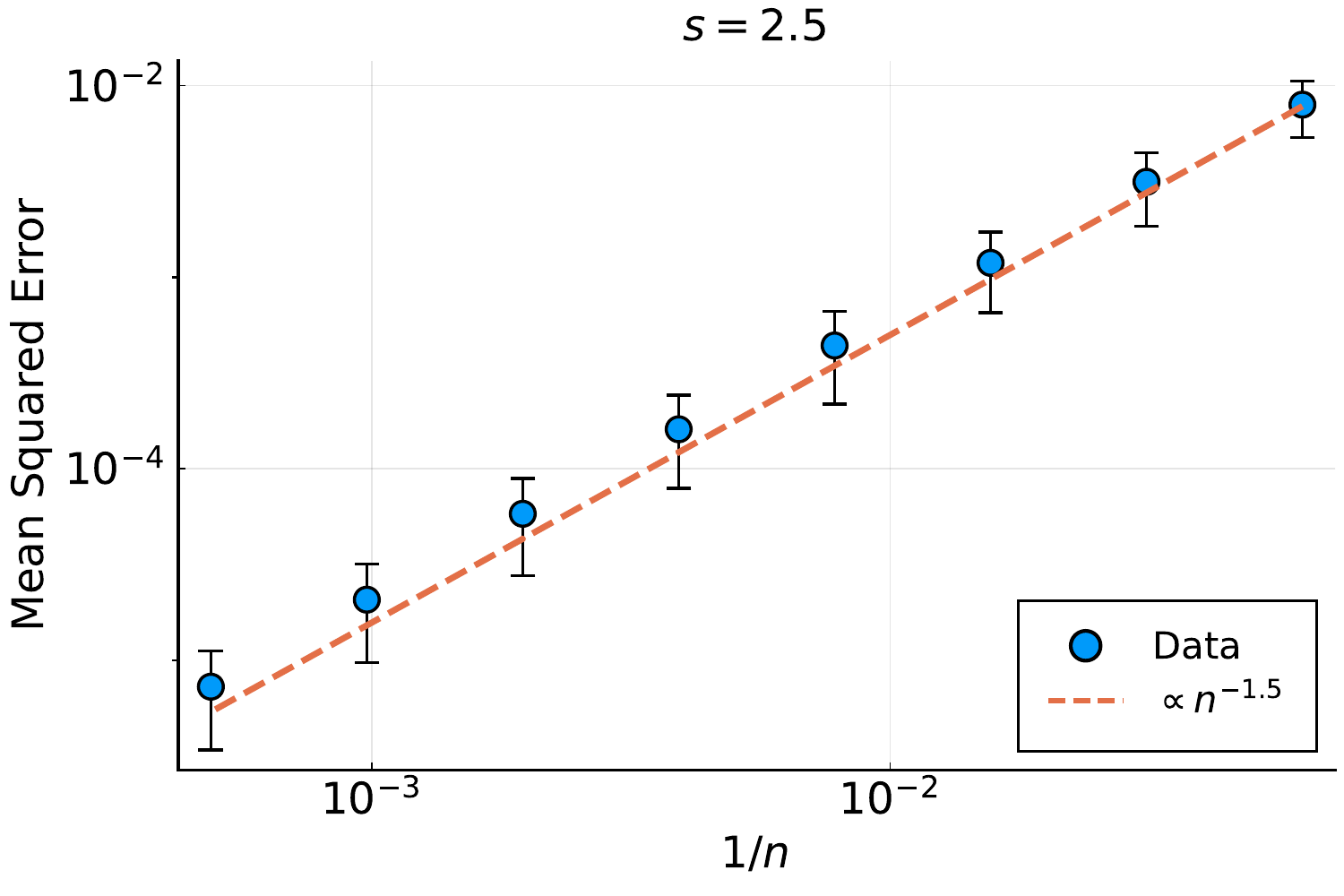}}
    \subfigure[]{\includegraphics[width=6.25cm]{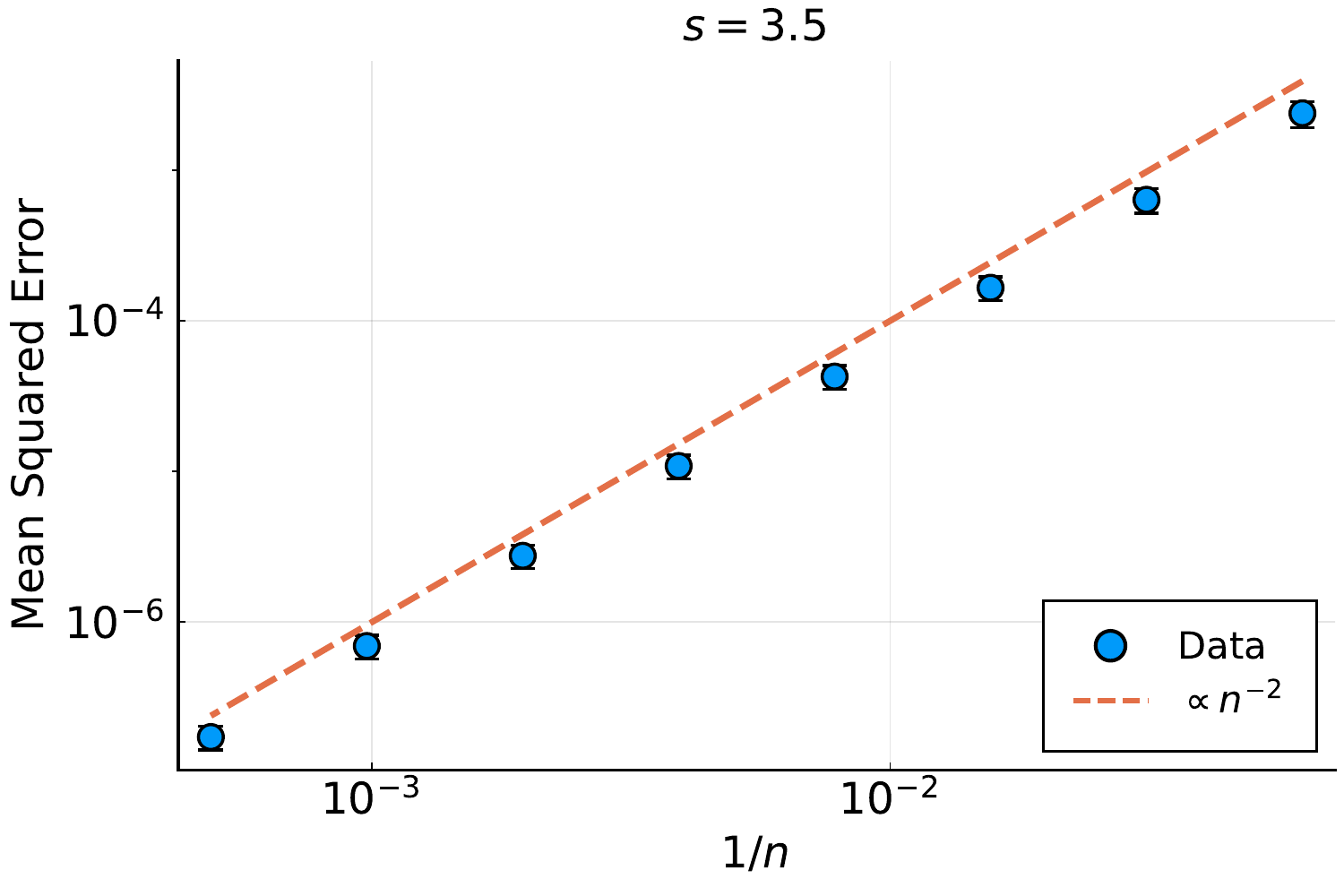}}
                
    \caption{Convergence of the mean square error as a function of $n$ at fixed time step $\Dt=0.001$.  The reference path is generated with $2^{13}$ sample points.  Error bars are one standard deviation from $10^2$ trials.}
    
    \label{fig:nconv1}
\end{figure}

At fixed $n = 1024$, we obtain the results shown in Figure \ref{fig:dtconv1}.  Here, we see the predicted $\propto \Dt$ error across all cases.

\begin{figure}
    \centering
    \subfigure[]{\includegraphics[width=6.25cm]{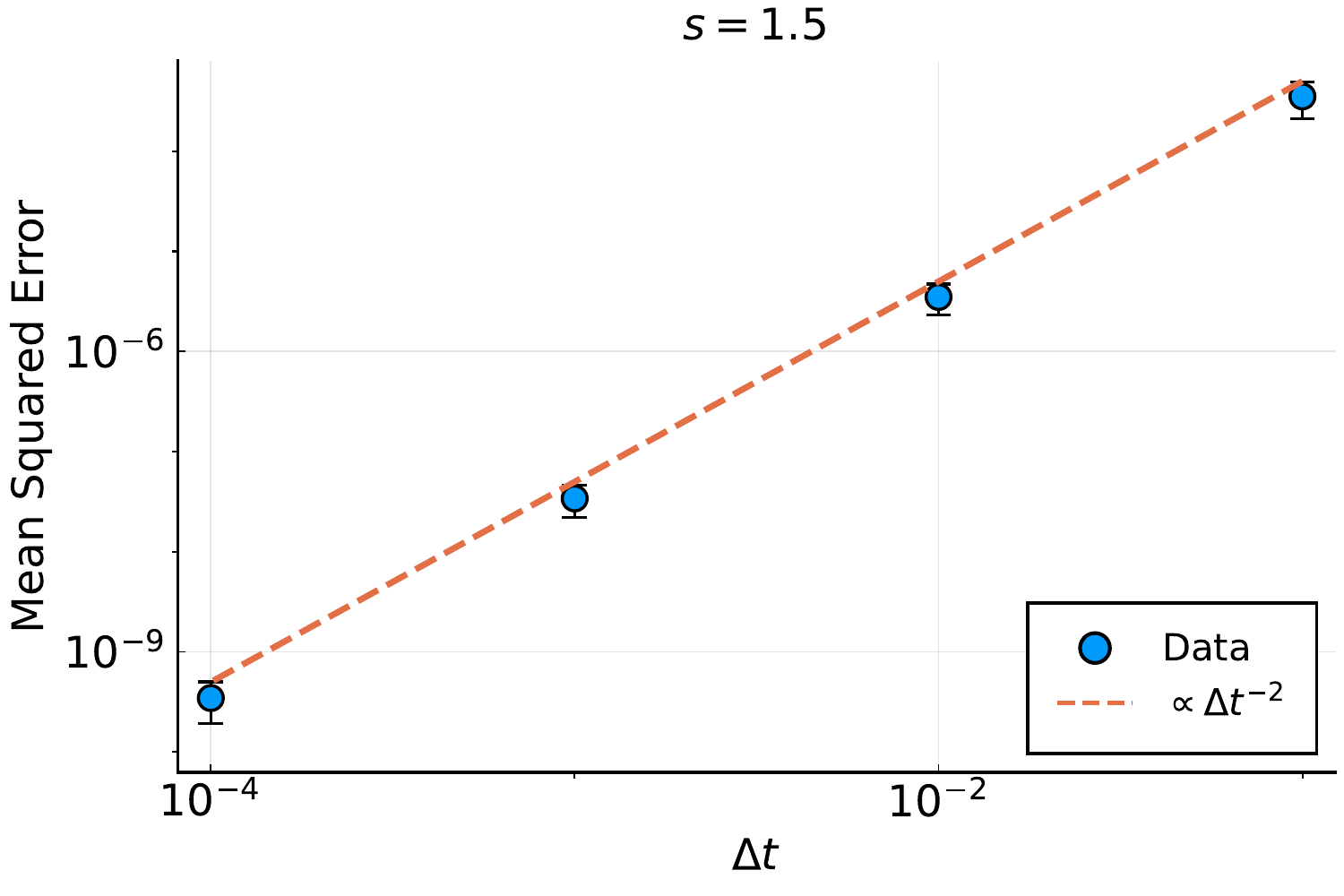}}
    \subfigure[]{\includegraphics[width=6.25cm]{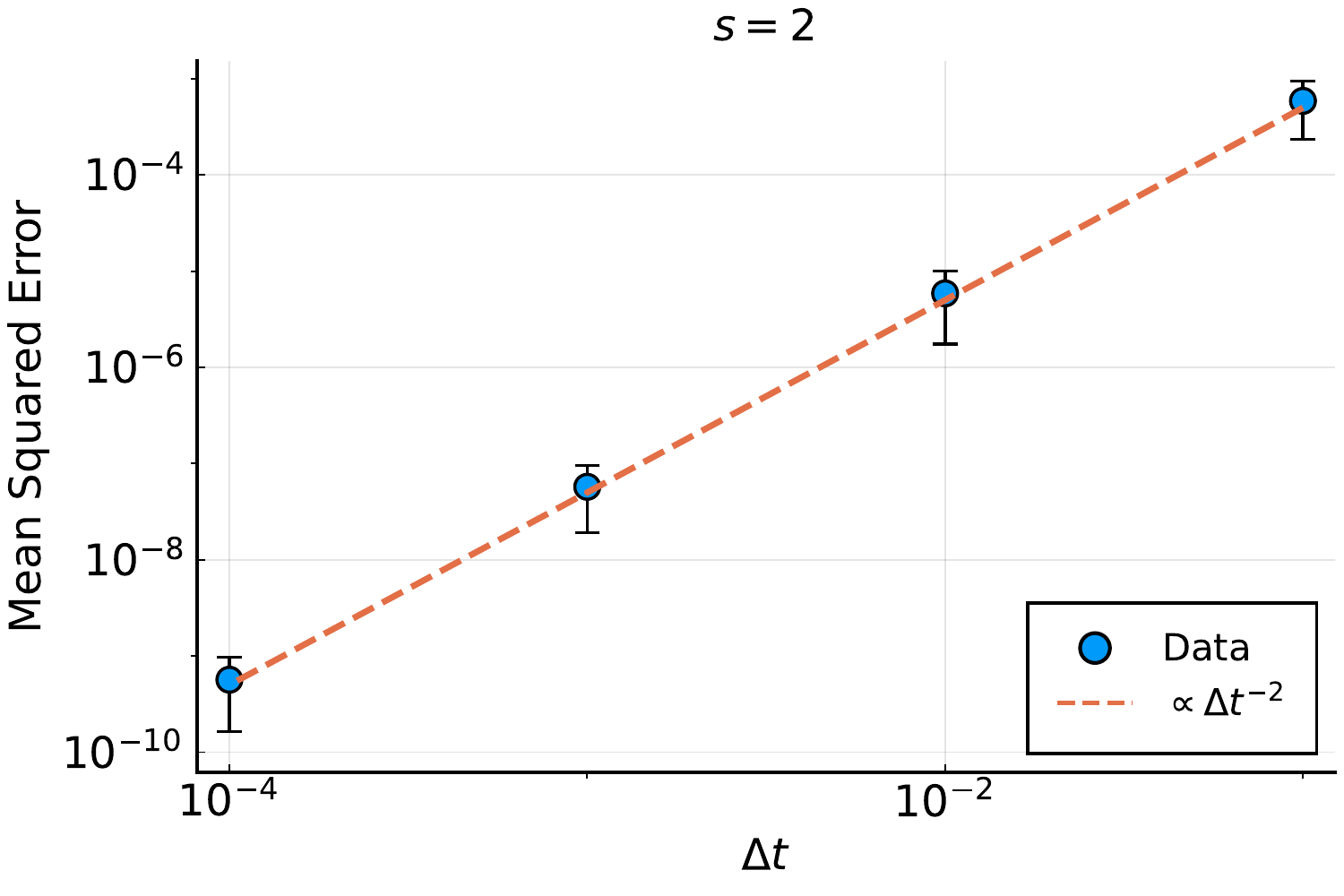}}
        
    \subfigure[]{\includegraphics[width=6.25cm]{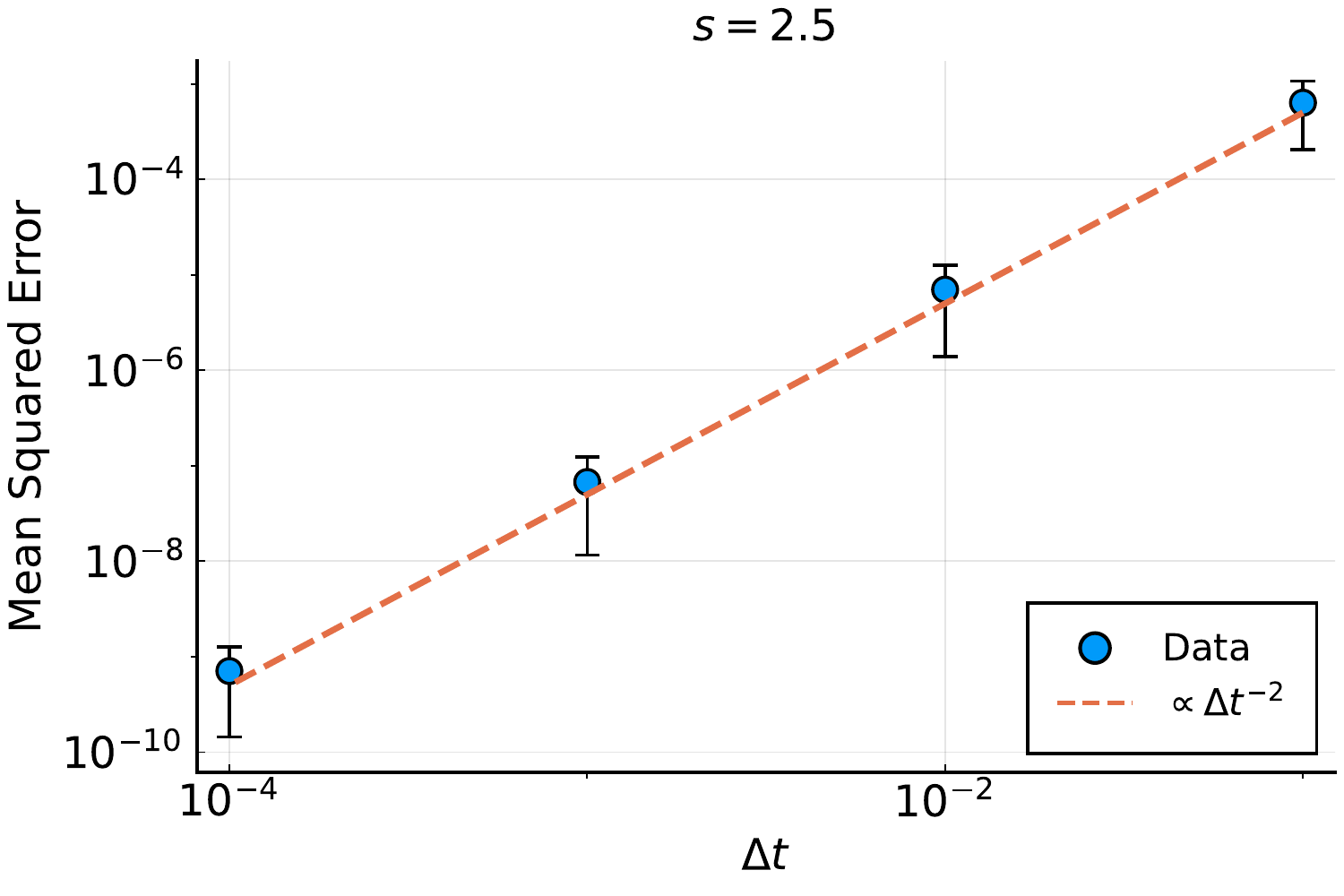}}
    \subfigure[]{\includegraphics[width=6.25cm]{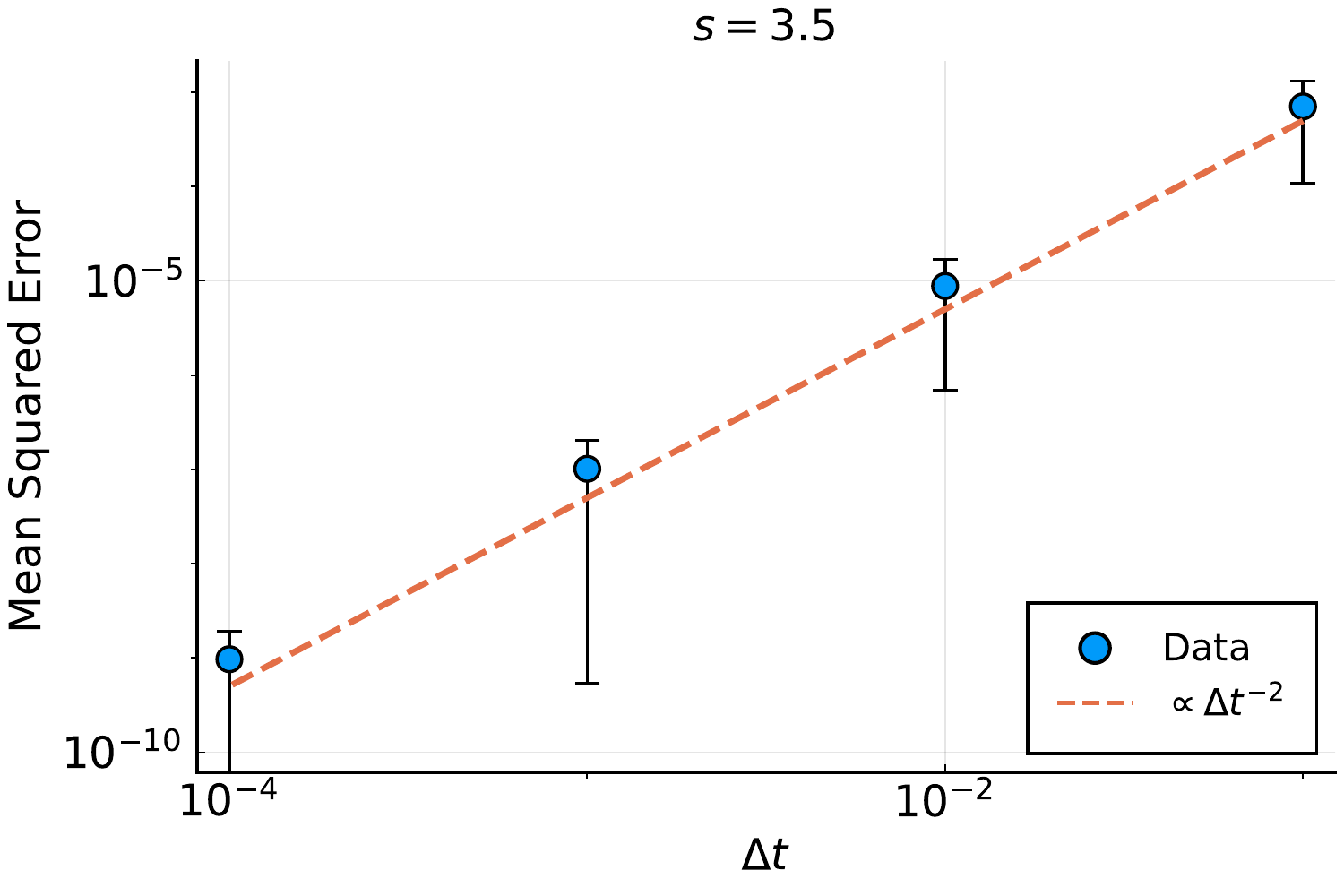}}
                
    \caption{Convergence of the mean square error as a function of $\Dt$ at fixed spatial resolution $n=1024$.  The reference path is generated with $\Dt = 10^{-6}$.  Error bars are one standard deviation from $10^2$ trials.}
    
    \label{fig:dtconv1}
\end{figure}

\subsection{Understanding the spatial error discrepancy}
\label{s:spatial}

Consider the slight simplification of \eqref{e:Kexample},
\begin{subequations}
\label{e:K01}
\begin{align}
    B_r & = \set{(x,y)\in [0,1]^2\mid |x-y|<r},\\
    K(x,y) & = 1_{B_r}(x,y),
\end{align}
\end{subequations}
for some $r \in (0,1)$. This is in function is in $\Lip(1/2, L^2(I^2))$ and, using the preceding estimates,  contributes an error term $\propto n^{-1/2}$. In our proof of Theorem \ref{t:rate}, we treated the error \rev{of $ \vvvert\mathbf{K}[u^n]-\mathbf{K}^n[u^n]\vvvert_{2}$ with the  $L^2(I^d\times I^d)$ error of $\bP_n^{(1)}K$; this appears in \eqref{K-Kn}.}  We could have, instead, bounded it in $L^1_y L^\infty_x$, to obtain
\begin{equation*}
     \trinorm{\mathbf{K}[u^n] - \mathbf{K}^n[u^n]}_2\lesssim \|\|K(x,\cdot) -(\rev{\mathbf{P}_n^{(1)}} K)(x,\cdot)\|_{L^1_y}\|_{L^\infty_x} .
\end{equation*}
This can give us higher order convergence.  Indeed, consider,  $x \in (r,1-r)$, and assume that $n$ is  sufficiently large that 
\[
r < x_{i-1}< x < x_i < 1-r.
\]
Then, by a geometric argument,
\begin{equation*}
\begin{split}
    \int |K(x,y) - (\rev{\mathbf{P}_n^{(1)}}K)(x,y)|dy &= \int_{x_{i-1}-r}^{x_{i} -r}  |K(x,y) - (\rev{\mathbf{P}_n^{(1)}}K)(x,y)|dy \\
    &\quad + \int_{x_{i-1}+r}^{x_{i}+r} |K(x,y) - (\rev{\mathbf{P}_n^{(1)}}K)(x,y)|dy \\
    & = \int_{x_{i-1}-r}^{x_{i} -r}  |1_{|x-y|<r} - \Delta x^{-1}(y-(x_{i-1}-r))|dy \\
    &\quad + \int_{x_{i-1}+r}^{x_{i}+r} |1_{|x-y|<r} - \Delta x^{-1}(x_{i}+r - y))|dy \\
    & \lesssim  \Delta x.
\end{split}
\end{equation*}
Similar arguments hold for when $x< r$ and when $x> 1-r$.  Consequently, 
\begin{equation*}
    \trinorm{\mathbf{K}[u^n] - \mathbf{K}^n[u^n]}_2\lesssim n^{-1}
\end{equation*}
instead of the $n^{-1/2}$ rate we would get from an $L^2(\rev{I^d\times I^d})$ analysis.

\section{Discussion}\label{s:discussion}
In this paper, we examined the well-posedness and analyzed a numerical method for a nonlocal evolution equation describing dynamics of interacting particles on graph forced by noise in the limit as the number of particles goes to infinity.  We found good agreement between our numerical experiments and the predictions, and we were able to explain  the discrepancy between the more general result, Theorem \ref{t:fulldiscrete}, and the experiments.

Several extensions are possible. First, it is straightforward to extend this algorithm and the error analysis to cover models with random initial data. Second, one can combine the Galerkin method  with the Monte Carlo approximation of the nonlocal term to reduce the amount of computation necessary to achieve a given degree accuracy (cf.~\cite{KVMed20}). This approach \rev{is} especially effective for models
with nonsmooth kernels and for higher dimensional spatial domains.  

Another extension would be to further develop the convergence analysis with respect to the interaction kernel, as discussed in Section \ref{s:spatial}.  There, we remarked that if the error were measured in the $L_y^1\otimes L^\infty_x$ norm, we could obtain higher order convergence than in the $L^2_x \otimes L^2_y$ norm.  It would be desirable to determine an ``optimal'' function space in which to study the projection error of the kernel.   Likewise, we found that for trigonometric nonlinearities, we could improve our time stepping error to match that of Milstein's method; this was the content of Theorem \ref{t:fulldiscrete2}.  It would also be desirable to identify the full class of nonlinear interactions, $S$, for which this higher order convergence holds.  A final extension of this work would be to allow for multiplicative, instead of additive, noise.

As a by-product, this work also presents a rigorous continuum limit for a large class of interacting dynamical systems on graphs subject to noise. Existing continuum models for interacting diffusions on graphs rely on Sznitman's nonlinear process framework \cite{Sznitman91}, which requires additional integration of McKean-Vlasov partial differential equation \cite{Lucon2020}. Thus, our model presents a simpler, and more direct, description of the continuum limit of interacting diffusions on graphs \rev{in the spirit
of \cite{Med19}}. At the technical level, we prove convergence of discrete models in stronger topology than that of weakly continuous measure valued process that is normally used in this context. In addition to providing continuum description for many common in applications such as the Kuramoto model  of coupled phase oscillators and discrete models of neural tissue, our method can be used for numerical integration of nonlocal diffusion equations, including nonlinear and fractional diffusion models.  Other applications include population dynamics, swarming, and peridynamics.

\vskip 0.2cm
\noindent
{\bf Acknowledgements.} This work was supported in part by NSF grant DMS-2009233 (to GSM)
and DMS-1818726\rev{, DMS-2111278} (to GS).

\rev{Work reported here was run on hardware supported by Drexel's University Research Computing Facility.}


\begin{thebibliography}{10}
\bibitem{AMRT10}
Fuensanta Andreu-Vaillo, Jos\'{e}~M. Maz\'{o}n, Julio~D. Rossi, and
  J.~Juli\'{a}n Toledo-Melero, \emph{Nonlocal diffusion problems}, Mathematical
  Surveys and Monographs, vol. 165, American Mathematical Society, Providence,
  RI; Real Sociedad Matem\'{a}tica Espa\~{n}ola, Madrid, 2010. 

\bibitem{Noch19}
Lehel Banjai, Jens~M. Melenk, Ricardo~H. Nochetto, Enrique Ot\'{a}rola,
  Abner~J. Salgado, and Christoph Schwab, \emph{Tensor {FEM} for spectral
  fractional diffusion}, Found. Comput. Math. \textbf{19} (2019), no.~4,
  901--962. 

\bibitem{BBPN18}
Andrea Bonito, Juan~Pablo Borthagaray, Ricardo~H. Nochetto, Enrique
  Ot\'{a}rola, and Abner~J. Salgado, \emph{Numerical methods for fractional
  diffusion}, Comput. Vis. Sci. \textbf{19} (2018), no.~5-6, 19--46.
  

\bibitem{BouCal10}
Nikolaos Bournaveas and Vincent Calvez, \emph{The one-dimensional
  {K}eller-{S}egel model with fractional diffusion of cells}, Nonlinearity
  \textbf{23} (2010), no.~4, 923--935. 

\bibitem{CarFife05}
C.~Carrillo and P.~Fife, \emph{Spatial effects in discrete generation
  population models}, J. Math. Biol. \textbf{50} (2005), no.~2, 161--188.
  

\bibitem{Cha17}
Sourav Chatterjee, \emph{Large deviations for random graphs}, Lecture Notes in
  Mathematics, vol. 2197, Springer, Cham, 2017, Lecture notes from the 45th
  Probability Summer School held in Saint-Flour, June 2015, \'{E}cole
  d'\'{E}t\'{e} de Probabilit\'{e}s de Saint-Flour. [Saint-Flour Probability
  Summer School].

\bibitem{CGP14}
Stephen Coombes, Peter beim Graben, and Roland Potthast, \emph{Tutorial on
  neural field theory}, Neural fields, Springer, Heidelberg, 2014, pp.~1--43.
 

\bibitem{da2014stochastic}
Giuseppe Da~Prato and Jerzy Zabczyk, \emph{Stochastic equations in infinite
  dimensions}, Cambridge university press, 2014.

\bibitem{PabQuiRod16}
Arturo de~Pablo, Fernando Quir\'{o}s, and Ana Rodr\'{i}guez, \emph{Nonlocal
  filtration equations with rough kernels}, Nonlinear Anal. \textbf{137}
  (2016), 402--425. 

\bibitem{TEJ17}
F\'{e}lix del Teso, J{\o}rgen Endal, and Espen~R. Jakobsen, \emph{Uniqueness and
  properties of distributional solutions of nonlocal equations of porous medium
  type}, Adv. Math. \textbf{305} (2017), 78--143.

\bibitem{Du2018}
Qiang Du, \emph{An invitation to nonlocal modeling, analysis, and computation}, Proceedings of the International Congress of Mathematicians: Rio de Janeiro 2018, 
  3541--3569, 2018.

\bibitem{Du2019b}
Qiang Du, Lili Ju, and Jianfang Lu, \emph{A discontinuous {G}alerkin method for
  one-dimensional time-dependent nonlocal diffusion problems}, Math. Comp.
  \textbf{88} (2019), no.~315, 123--147. 

\bibitem{HOH2010}
Harald Hanche-Olsen and Helge Holden, \emph{The {K}olmogorov-{R}iesz
  compactness theorem}, Expo. Math. \textbf{28} (2010), no.~4, 385--394.
  

\bibitem{KVMed20}
Dmitry {Kaliuzhnyi-Verbovetskyi} and Georgi~S. {Medvedev}, \emph{{Sparse Monte
  Carlo method for nonlocal diffusion problems}}, arXiv e-prints (2019),
  arXiv:1905.10844.

\bibitem{kloeden2013numerical}
Peter~E Kloeden and Eckhard Platen, \emph{Numerical solution of stochastic
  differential equations}, vol.~23, Springer Science \& Business Media, 2013.

\bibitem{Laing17}
Carlo~R. Laing, \emph{Chimeras in two-dimensional domains: heterogeneity and
  the continuum limit}, SIAM J. Appl. Dyn. Syst. \textbf{16} (2017), no.~2,
  974--1014. 

\bibitem{liu2015stochastic}
Wei Liu and Michael R{\"o}ckner, \emph{Stochastic partial differential
  equations: an introduction}, Springer, 2015.

\bibitem{Lord.2018}
Gabriel~J. Lord and Antoine Tambue, \emph{{A modified semi–implicit
  Euler–Maruyama scheme for finite element discretization of SPDEs with
  additive noise}}, Applied Mathematics and Computation \textbf{332} (2018),
  105--122.

\bibitem{lord2014an}
GJ~Lord, CE~Powell, and T~Shardlow, \emph{An introduction to computational
  stochastic pdes}, Cambridge University Press, 2014.

\bibitem{Lucon2020}
Eric Lu\c{c}on, \emph{Quenched asymptotics for interacting diffusions on
  inhomogeneous random graphs}, Stochastic Process. Appl. \textbf{130} (2020),
  no.~11, 6783--6842. 

\bibitem{Med19}
Georgi~S. Medvedev, \emph{The continuum limit of the {K}uramoto model on sparse
  random graphs}, Communications in Mathematical Sciences \textbf{17} (2019),
  no.~4, 883--898.

\bibitem{Med14a}
\bysame, \emph{The nonlinear heat equation on dense graphs and graph
  limits}, SIAM J. Math. Anal. \textbf{46} (2014), no.~4, 2743--2766.
  

\bibitem{Med14b}
\bysame, \emph{The nonlinear heat equation on {W}-random graphs}, Arch. Ration.
  Mech. Anal. \textbf{212} (2014), no.~3, 781--803. 

\bibitem{MotTad14}
Sebastien Motsch and Eitan Tadmor, \emph{Heterophilious dynamics enhances
  consensus}, SIAM Rev. \textbf{56} (2014), no.~4, 577--621. 

\bibitem{NochOtaSal16}
Ricardo~H. Nochetto, Enrique Ot\'{a}rola, and Abner~J. Salgado, \emph{A {PDE}
  approach to space-time fractional parabolic problems}, SIAM J. Numer. Anal.
  \textbf{54} (2016), no.~2, 848--873. 

\bibitem{Noch16}
\bysame, \emph{A {PDE} approach to space-time fractional parabolic problems},
  SIAM J. Numer. Anal. \textbf{54} (2016), no.~2, 848--873. 

\bibitem{Sznitman91}
Alain-Sol Sznitman, \emph{Topics in propagation of chaos}, \'{E}cole
  d'\'{E}t\'{e} de {P}robabilit\'{e}s de {S}aint-{F}lour {XIX}---1989, Lecture
  Notes in Math., vol. 1464, Springer, Berlin, 1991, pp.~165--251. 

\bibitem{Vazquez17}
Juan~Luis V\'{a}zquez, \emph{The mathematical theories of diffusion: nonlinear
  and fractional diffusion}, Nonlocal and nonlinear diffusions and
  interactions: new methods and directions, Lecture Notes in Math., vol. 2186,
  Springer, Cham, 2017, pp.~205--278. 

\bibitem{Wang.2015}
Xiaojie Wang and Ruisheng Qi, \emph{{A note on an accelerated exponential Euler
  method for parabolic SPDEs with additive noise}}, Applied Mathematics Letters
  \textbf{46} (2015), 31--37.
\end{thebibliography}
\providecommand{\bysame}{\leavevmode\hbox to3em{\hrulefill}\thinspace}
\providecommand{\MR}{\relax\ifhmode\unskip\space\fi MR }
\providecommand{\MRhref}[2]{%
  \href{http://www.ams.org/mathscinet-getitem?mr=#1}{#2}
}
\providecommand{\href}[2]{#2}

\appendix

\section{Supplementary Calculations}
\label{s:supplement}

\begin{lemma}
\label{l:EsinY}
Let $Y$ be an $\mathcal{H} = L^2(I^d)$ valued Gaussian random variable with mean zero and trace class covariance operator ${\bf Q}$, with eigenvalues $\lambda_k$ and eigenfunctions $e_k$.  Then $\mathbb{E}[\sin(Y)]=0$
\end{lemma}
\begin{proof}

\begin{enumerate}[label=\arabic*., leftmargin=\parindent]

\item We first write  $Y$ using the  a Karhunen-Lo\`eve representation,
$$
Y = \sum_{k=1}^\infty \sqrt{\lambda_k} \xi_k e_k,
$$
and truncate it to the first $N$ modes,
$$
Y_N = \sum_{k=1}^N\sqrt{\lambda_k} \xi_k e_k.
$$
$Y_N\to Y$ in $L^2(\P;\rev{\cH})$, as
$$
\E[\|Y-Y_N\|^2] = \sum_{k=N+1}^\infty \lambda_k.
$$
As $\Tr {\bf Q} < \infty$, this clearly vanishes.  Suppose we can show that, for all $N$, $\E[\sin(Y_N)]=0$.  Then, for any $N$,
\begin{equation*}
\begin{split}
\|\E[\sin(Y)]\|^2 &= \|\E[\sin(Y)] - \E[\sin(Y_N)]\|^2\\
&\leq \E[\|\sin(Y) - \sin(Y_N)\|^2]\leq \E[\|Y - Y_N\|^2]
\end{split}
\end{equation*}
Since this vanishes as $N\to \infty$, $\E[\sin(Y)] = 0$ with equality in the sense of $L^2$.

\item Next, we verify that for any $N$, $\E[\sin(Y_N)]=0$.  Let
$$
Y_{N,\eps} = \sum_{k=1}^N\sqrt{\lambda_k} \xi_k \varphi_{k,\eps}.
$$
where $\varphi_{k,\eps}$ are mollified eigenfunctions so as to allow for pointwise evaluation.  Since $N$ is finite, we can be assured that $\varphi_{k,\eps}\to e_k$ in $L^2$ as $\eps\to 0$, uniformly in $k=1,\ldots,N$.  We will verify that for any $x$ and any $\eps$, $\E[\sin(Y_{N,\eps}(x))]=0$.  Consequently,
\begin{equation*}
\begin{split}
\|\E[\sin(Y_N)]\|^2 &= \|\E[\sin(Y_N) - \sin(Y_{N,\eps})\|^2\\
&\leq \E[\|\sin(Y_N) - \sin(Y_{N,\eps})\|^2]\\
&\leq \E[\|Y_N - Y_{N,\eps}\|^2]\\
&\leq \sum_{k=1}^N \lambda_k \|e_k - \varphi_{k,\eps}\|^2\leq \lambda_1 \sum_{k=1}^N\|e_k - \varphi_{k,\eps}\|^2.
\end{split}
\end{equation*}
This obviously vanishes as $\eps\to 0$.

\item Finally, for any $x$ and any $\eps$, $Y_{N,\eps}(x)$ is a scalar mean zero Gaussian with variance
$$
\sum_{k=1}^N \lambda_k \varphi_{k,\eps}(x)^2< \infty.
$$
For such a random variable is a straightforward calculation to verify that $\E[\sin(Y_{N,\eps}(x))]=0$.

\end{enumerate}
\end{proof}

The following proposition shows that the bounds in Proposition \ref{p:Milstein} hold for a particular case of \eqref{e:kuramoto}, allowing us to obtain higher order convergence in time when Euler-Maruyama time stepping is used; see Theorem \ref{t:fulldiscrete2}.
\begin{proposition}
\label{p:MilsteinKuramoto}
Let $u$ solve \eqref{e:kuramoto} with $f=0$ and
\begin{equation*}
\mathbf{K}[u(t)] = \int K(x,y)\sin(2\pi(u(x,t)-u(y,t)))dy.
\end{equation*}
Then for any partition $0=t_0<t_1<\ldots< t_{M}=T$, $s \in [t_j, t_{j+1}]$
\begin{equation*}
\mathbf{K}[u(s)]     - \mathbf{K}[u(t_j)] = a_j(s) + \beta_j(s),
\end{equation*}
with $a_j$ and $\beta_j$ that satisfy the conditions:
\begin{gather*}
    \trinorm{a_j(s)}_2 \lesssim(s-t_j),\\
    \trinorm{\beta_j(s)}_2 \lesssim \sqrt{s-t_j}, \quad \E[\beta_j(s)\mid \mathcal{F}_{t_j}] = 0\rev{,}
\end{gather*}
\rev{where the implicit constants are independent of the $t_j$.}
\end{proposition}
The precise form of $a_j$ and $\beta_j$ is not essential, but it can be found below in \eqref{e:aj} and \eqref{e:betaj}.  

\begin{proof}
\begin{enumerate}[label=\arabic*., leftmargin=\parindent]
 We begin by writing
\begin{equation}
\label{e:Ksplit1}
\begin{split}
    &\mathbf{K}[u(s)] - \mathbf{K}[u(t_j)]\\
    &= \int K(x,y)\sin(2\pi(u(x,s)-u(y,s)))dy\\
    &\quad - \int K(x,y)\sin(2\pi(u(x,t_j)-u(y,t_j)))dy.
\end{split}
\end{equation}
It will be sufficient to analyze one of the integral terms.
\item Define the following terms to simplify the expressions
\begin{align}
\Delta_j u(x,s) &= u(x,s) - u(x,t_j)\\
\delta_{xy} u(s) & = u(x,s) - u(y,s)\\
\Delta_j \delta_{xy} u(s) &= (u(x,s) - u(y,s)) - (u(x,t_j) - u(y,t_j))
\end{align}
We will occasionally suppress the $x$ or $y$ dependence when there is no ambiguity.  Then one of the integrand terms in \eqref{e:Ksplit1} is
\begin{equation*}
    \begin{split}
        \sin(2\pi\delta_{xy}u(s)) - \sin(2\pi \delta_{xy}u(t_j))
        & = \sin(2\pi \delta_{xy}u(t_j))[\cos(2\pi \rev{\Delta_j\delta_{xy}u(s)})-1]\\
        &\quad + \cos(2\pi\delta_{xy}u(t_j))\underbrace{\sin(2\pi\Delta_j \delta_{xy}\rev{u(s)} ))}_{\equiv I}
    \end{split}
\end{equation*}

\item Next, \rev{since}
\begin{equation*}
\begin{split}
    \Delta_j u(s) &= \underbrace{\int_{t_j}^{s}\mathbf{K}[u(\tau)]d\tau}_{\equiv \Delta_j F(s) } +  (\underbrace{W(s) -  W(t_j)}_{\equiv \Delta_j W(s)})
\end{split}
\end{equation*}
with analogous expressions for $\Delta_j\delta_{xy}F(s)$ and $\Delta_j\delta_{xy}W(s)$, we write
\rev{\begin{equation*}
\begin{split}
I&= \sin(2\pi\Delta_j\delta_{xy} F(s) )\cos(2\pi\Delta_j\delta_{xy}W(s) )+\cos(2\pi\Delta_j\delta_{xy} F(s))\sin(2\pi\Delta_j \delta_{xy}W(s) )\\
    &=\sin(2\pi\Delta_j\delta_{xy} F(s) )\cos(2\pi\Delta_j\delta_{xy}W(s) ) + [\cos(2\pi\Delta_j\delta_{xy}F(s) )-1]\sin(2\pi\Delta_j\delta_{xy} W(s))\\
    &\quad + \underbrace{\sin(2\pi\Delta_j \delta_{xy}W(s) )}_{\equiv II}
\end{split}
\end{equation*}}
\item Finally, we expand the last term, to obtain
\begin{equation*}
\begin{split}
II& = \sin(2\pi\Delta_j W(x,s))\cos(2\pi\Delta_j W(y,s))\\
&\quad- \cos(2\pi\Delta_j W(x,s))\sin(2\pi\Delta_j W(y,s))\\
&=\sin(2\pi\Delta_j W(x,s))\\
&\quad + \sin(2\pi\Delta_j W(x,s))[\cos(2\pi\Delta_j W(y,s))-1] \\
&\quad  -\sin(2\pi\Delta_j W(y,s)) \\
&\quad - [\cos(2\pi\Delta_j W(x,s))-1]\sin(2\pi\Delta_j W(y,s))
\end{split}
\end{equation*}
\item The original nonlinear interaction term in \eqref{e:Ksplit1} can now be expressed as
\begin{equation*}
    \begin{split}
        &\sin(2\pi\delta_{xy}u(s)) - \sin(2\pi\delta_{xy}u(t_j))\\
        &=\theta_j^{(1)}(x,y) [\cos(2\pi\Delta_j \delta_{xy}u(s))-1]+ \theta_j^{(2)}(x,y,s)\sin(2\pi\Delta_j \delta_{xy}F(s) )\\
        &\quad + \theta_j^{(3)}(x,y,s)  [\cos(2\pi(\Delta_j\delta_{xy}F(s))-1]\\
        &\quad +\theta_j^{(4)}(x,y,s)[\cos(2\pi\Delta_j W(y,s))-1]\\
        &\quad - \theta_j^{(5)}(x,y,s) [\cos(2\pi\Delta_j W(x,s))-1]\\
        &\quad +\eta_j(x,y) (\sin(2\pi\Delta_j W(x,s))-\sin(2\pi\Delta_j W(y,s)))
    \end{split}
\end{equation*}
where
\begin{subequations}
    \begin{align*}
        \theta_j^{(1)} & =\sin(2\pi\delta_{xy}u(t_j))\\
        \theta_j^{(2)} & =\cos(2\pi\delta_{xy}u(t_j)) \cos(2\pi\Delta_j \delta_{xy}W(s))\\
        \theta_j^{(3)} & =\cos(2\pi\delta_{xy}u(t_j)) \sin(2\pi\Delta_j \delta_{xy}W(s))\\
        \theta_j^{(4)} & =\cos(2\pi \delta_{xy}u(t_j)) \sin(2\pi\Delta_j W(x,s)) \\
        \theta_j^{(5)} & =\cos(2\pi\delta_{xy}u(t_j))\sin(2\pi\Delta_j W(y,s)) \\
        \eta_j & = \cos(2\pi\delta_{xy}u(t_j))
    \end{align*}
\end{subequations}
\item The terms that we need to analyze to reach our result, $a_j$ and $\beta_j$, are now given explicitly.
\begin{equation}
\label{e:aj}
\begin{split}
    a_j(s) &= \int K(\cdot,y) \left\{\theta_j^{(1)}(\cdot,y) [\cos(2\pi\Delta_j\delta_{xy}u(s))-1]\right.\\
    &+ \theta_j^{(2)}(\cdot,y, s) \sin(2\pi\Delta_j \delta_{xy}F(s)) \\
    &+ \theta_j^{(3)}(\cdot,y, s)  [\cos(2\pi\Delta_j \delta_{xy}F(s)  )-1]\\
    &+\theta_j^{(4)}(\cdot,y, s) [\cos(2\pi\Delta_j W(\cdot,s))-1]\\
    &\left. \rev{-}\theta_j^{(5)}(\cdot,y, s) [\cos(2\pi\Delta_j W(y,s))-1]\right\}dy
\end{split}
\end{equation}
while
\begin{equation}
\label{e:betaj}
    \beta_j(s)  = \int K(\cdot,y) \eta_j(\cdot,y) \{\sin(2\pi\Delta_j W(x,s))-\sin(2\pi\Delta_j W(y,s))\}dy
\end{equation}

\item We show that $a_j$ has the desired property.   First, 
\begin{equation*}
    \vvvert a_j(s)\vvvert_2 \leq \sum_{k=1}^5 A_j^{(k)}(s)
\end{equation*}
where
\rev{\begin{align*}
    A_j^{(1)}(s)&=\trinorm{\int K(\cdot,y)\theta_j^{(1)}(\cdot,y) [\cos(2\pi\Delta_j\delta_{xy}u(s))-1] dy }_2\\
    A_j^{(2)}(s)&={\trinorm{\int K(\cdot,y)\theta_j^{(2)}(\cdot,y, s) \sin(2\pi\Delta_j \delta_{xy}F(s)) dy }_2}\\
    A_j^{(3)}(s)&=\trinorm{\int K(\cdot,y)\theta_j^{(3)}(\cdot,y, s)  [\cos(2\pi\Delta_j \delta_{xy}F(s)  )-1] dy }_2\\
    A_j^{(4)}(s)&=\trinorm{\int K(\cdot,y)\theta_j^{(4)}(\cdot,y, s) [\cos(2\pi\Delta_j W(\cdot,s))-1] dy }_2\\
    A_j^{(5)}(s)&=\trinorm{\int K(\cdot,y)\theta_j^{(5)}(\cdot,y, s) [\cos(2\pi\Delta_j W(y,s))-1] dy }_2
\end{align*}}
\item We now show for each $k$,  $\vvvert A_j^{(k)}(s)\vvvert_2 \lesssim (s-t_j)$, with \rev{ a constant that is independent of the $t_j$}.  This relies on the elementary inequalities:
\begin{align*}
    |\sin(x)|&\leq |x|\\
    |\cos(x)-1|&\leq |x|\\
    |\cos(x)-1|& \leq \tfrac{1}{2}|x|^2
\end{align*}
First,
\begin{equation*}
\begin{split}
    (A_j^{(1)}(s))^2& = \E\bracket{\int\abs{\int K(x,y) \theta_j^{(1)}(x,y)[\cos(2\pi\Delta_j \delta_{xy}u(s)) - 1]dy}^2dx }\\
    &\leq \E\bracket{\iint |K(x,y)|^2 |\theta^{(1)}_j(x,y)|^2[\cos(2\pi\Delta_j \delta_{xy}u(s))-1]^2dydx}\\
    &\lesssim \E\bracket{\iint |K(x,y)|^2 (|\Delta_ju(x,s)|^2 + |\Delta_ju(y,s)|^2)^2dxdy}\\
    &\lesssim (\|\|K(x,\cdot)\|_{L^2}\|_{L^\infty}^2+\|\|K(\cdot,y)\|_{L^2}\|_{L^\infty}^2)\E[\|\Delta u_j(s)\|^4]
\end{split}
\end{equation*}
Consequently, by Corollary \ref{c:cont},
\begin{equation*}
    A_j^{(1)}\lesssim \trinorm{u(s) - u(t_j)}_{4}^2\lesssim s - t_j
\end{equation*}
Similarly,
\begin{equation*}
\begin{split}
(A_j^{(2)}(s))^2  &= \E\bracket{\int\abs{\int K(x,y) \theta_j^{(2)}(x,y,s) \sin(2\pi\Delta_j \delta_{xy}F(s) )dy}^2dx}\\
&\leq \E\bracket{\iint |K(x,y)|^2 |\sin(2\pi\Delta_j \delta_{xy}F(s) )|^2 dydx}\\
&\lesssim \E\bracket{(\|\|K(x,\cdot)\|_{L^2}\|_{L^\infty}^2+\|\|K(\cdot,y)\|_{L^2}\|_{L^\infty}^2)|\Delta_j F(x,s)|^2 dx}\\
&\lesssim \E[\|\Delta_j F(s)\|^2]
\end{split}
\end{equation*}
Since the \rev{trigonometric interaction term is bounded}
\begin{equation*}
 \E[\|\Delta_j F(s)\|^2] = \E\bracket{\norm{\int_{t_j}^s \mathbf{K}[u(\tau)] \tau}^2_{L^2}}\lesssim (s-t_j)^2
\end{equation*}
and we conclude $A_j^{(2)}(s)\lesssim s-t_j$. The term $A_j^{(3)}$ is established \rev{in} the same way as $A_j^{(2)}$, but using the estimate $|\cos(x)-1|\leq |x|$.
For $A_j^{(4)}$,
\rev{\begin{equation*}
\begin{split}
(A_j^{(4)})^2 & = \E\bracket{\int \abs{\int K(x,y)\theta_j^{(4)}(x,y,s) [\cos(2\pi \Delta_j W(y,s))-1] dy}^2dx}\\
&\lesssim \E\bracket{\iint |K(x,y)|^2 |\Delta_j W(y,s)|^4 dydx}\\
&\lesssim \E[\| W(s) - W(t_j)\|^4]
\end{split}
\end{equation*}}
\rev{Using the  properties of $W$,}
\begin{equation*}
    A_j^{(4)}\lesssim \trinorm{W(s) - W(t_j\rev{)}}_4^2 \lesssim s -t_j
\end{equation*}
$A_j^{(5)}$ is proved in the same way, and we have that $\vvvert a_j(s)\vvvert_2 \lesssim s- t_j$.

\item \rev{Conditioning, we examine}  the $\beta_j$ term:
\begin{equation*}
    \E[\beta_j\mid \mathcal{F}_{t_j}] = \int K(\cdot, y)\E[\eta_j(\cdot,y) \{\sin(2\pi\Delta_j W(\cdot,s))-\sin(2\pi\Delta_j W(y,s))\}\mid \mathcal{F}_{t_j}]dy.
\end{equation*}
Recall, $\eta_j = \cos(2\pi\delta_{xy}u(t_j))$, so it is $\mathcal{F}_{t_j}$ measurable, \rev{and}:
\begin{equation*}
\begin{split}
    &\E[\eta_j(x,y) \{\sin(2\pi\Delta_j W(x,s))-\sin(2\pi\Delta_j W(y,s))\}\mid \mathcal{F}_{t_j}]\\
    &= \eta_j(x,y)\E[ \sin(2\pi\Delta_j W(x,s))\mid \mathcal{F}_{t_j}]\\
    &\quad - \eta_j(x,y)\E[\sin(2\pi\Delta_j W(y,s))\mid \mathcal{F}_{t_j}]\\
    & = 0 - 0, \quad \text{a.s.}
\end{split}
\end{equation*}
by Lemma \ref{l:EsinY}.

Finally,
\rev{\begin{equation*}
\begin{split}
    \vvvert \beta_j(s) \vvvert_2^2 &\leq  \E\bracket{\iint |K(x,y)|^2 |\eta_i(x,y)|^2 |\sin(2\pi  \Delta_i W(\cdot,s)|^2dxdy} ds\\
    &\lesssim  \E[\|\Delta_j W(s)\|^2 ]ds \lesssim s - t_j
\end{split}
\end{equation*}}
where we have used \rev{the properties of $W(t)$ and that $\eta_j$ is bounded by one}.
\end{enumerate}
\end{proof}

\end{document}